\RequirePackage{fix-cm}
\documentclass[smallextended]{svjour3}       
\smartqed  
\usepackage{csquotes}
\usepackage{babel}
\usepackage{amsmath,amssymb}
\usepackage[utf8]{inputenc}
\usepackage{lscape}
 \def\bibhang{24pt}%
\usepackage{algorithm}
\usepackage{algorithmic}
\usepackage{comment}
\usepackage{booktabs}
\usepackage{multirow}
\usepackage{subfigure}
\usepackage[colorlinks=true,breaklinks=true,bookmarks=true,urlcolor=blue,
     citecolor=blue,linkcolor=blue,bookmarksopen=false,draft=false]{hyperref}
\usepackage{graphicx}
\def\EMAIL#1{\href{mailto:#1}{#1}}

\newcommand\K{\mathbb{K}}

\DeclareMathAlphabet{\mathpzc}{OT1}{pzc}{m}{it}

\numberwithin{equation}{section}
\numberwithin{proposition}{section}
\numberwithin{theorem}{section}
\numberwithin{remark}{section}
\numberwithin{definition}{section}
\numberwithin{example}{section}
\numberwithin{corollary}{section}
\usepackage{hyperref}
\usepackage{geometry}
\begin{document}

\title{Scholtes relaxation method for pessimistic bilevel optimization}

\titlerunning{Pessimistic bilevel optimization}        
\author{\small Imane Benchouk \and Lateef Jolaoso \and Khadra Nachi \and Alain Zemkoho}


\institute{Imane Benchouk \at
              Laboratory of Mathematical Analysis and Applications, University Oran 1, Algeria,             \EMAIL{benchouk.imane@univ-oran1.dz}           
           \and
          Lateef Jolaoso \at
           School of Mathematical Sciences,	University of Southampton, United Kingdom,
           \EMAIL{l.o.jolaoso@soton.ac.uk}
           \and
          Khadra Nachi \at
              Laboratory of Mathematical Analysis and Applications, University Oran 1, Algeria,
              \EMAIL{nachi.khadra@univ-oran1.dz}\\
        (Part of the work was completed while this author was visiting the School of Mathematics at the University of Southampton)
              \and
         Alain Zemkoho \at
           School of Mathematical Sciences,	University of Southampton, United Kingdom,
           \EMAIL{a.b.zemkoho@soton.ac.uk}.
}

\date{Received: date / Accepted: date}

\maketitle

\begin{abstract}
When the lower-level optimal solution set-valued mapping of a bilevel optimization problem is not single-valued, we are faced with an ill-posed problem, which gives rise to the optimistic and pessimistic bilevel optimization problems, as tractable algorithmic frameworks. However, solving the pessimistic bilevel optimization problem is far more challenging than the optimistic one; hence, the literature has mostly been dedicated to the latter class of the problem. The Scholtes relaxation has appeared to be one of the simplest and most efficient ways to solve the optimistic bilevel optimization problem in its Karush-Kuhn-Tucker (KKT) reformulation or the corresponding more general mathematical program with complementarity constraints (MPCC). Inspired by such a success, this paper studies the potential of the Scholtes relaxation in the context of the pessimistic bilevel optimization problem. To proceed, we consider a pessimistic bilevel optimization problem, where all the functions involved are at least continuously differentiable. Then assuming that the lower-level problem is convex, the KKT reformulation of the problem is considered under the Slater constraint qualification. Based on this KKT reformulation, we introduce the corresponding version of the Scholtes relaxation algorithm. {{We then construct theoretical results ensuring that the limit of a sequence of global/local optimal solutions (resp. {{stationary }} points) of the aforementioned Scholtes relaxation is a global/local optimal solution (resp. {{stationary }} point) of the KKT reformulation of the pessimistic bilevel program.}} {{The results are accompanied by technical constructions ensuring that the Scholtes relaxation algorithm is well-defined or that the corresponding parametric optimization problem is more tractable}}. Furthermore, we perform some numerical experiments to assess the performance of the Scholtes relaxation algorithm using various examples. In particular, we study the effectiveness of the algorithm in obtaining solutions that can satisfy the corresponding C-stationarity concept.
\end{abstract}

\section{Introduction}\label{Introduction}
We consider the pessimistic bilevel optimization problem
\begin{equation}\tag{P$_p$}\label{PBP}
    \underset{x\in X}\min~\varphi_{p}(x):=\underset{y\in S(x)}\max~F(x,y),
\end{equation}
where $X:=\left\{x\in \mathbb{R}^n\left|\; G(x)\leq 0\right.\right\}$ and $F\, : \mathbb{R}^n\times \mathbb{R}^m \rightarrow \mathbb{R}$ represent the upper-level/leader's feasible set and objective function, respectively. The upper-level constraint function $G$ is defined from $\mathbb{R}^n$ to $\mathbb{R}^p$. As for the set-valued mapping $S: \mathbb{R}^n \rightrightarrows \mathbb{R}^m$, it collects the optimal solutions of the lower-level/follower's problem for selections of the leader; i.e., for all $x\in X$, we have
 \begin{equation}\label{S(x)}
    S(x):=\underset{y\in K(x)}{\arg\min}~f(x,y)
 \end{equation}
 and assume throughout that  $S(x) \neq \varnothing$ if $x\in X$ and $S(x)= \varnothing$ otherwise. Note that in \eqref{S(x)}, the function $f\, : \mathbb{R}^n\times \mathbb{R}^m \rightarrow \mathbb{R}$ represents the lower-level/follower's objective function, while the set-valued mapping $K: \mathbb{R}^n \rightrightarrows \mathbb{R}^m$ describes the lower-level/follower's feasible points defined by
 \begin{equation*}\label{M(x)}
    K(x):=\left\{y\in \mathbb{R}^m|\; g(x,y)\leq 0\right\}
 \end{equation*}
with $g\, : \mathbb{R}^n\times \mathbb{R}^m \rightarrow \mathbb{R}^q$. We assume  that  $F$, $f$, $G$, and $g$ are twice continuously differentiable. More general functions and feasible sets can be considered for our analysis. But to focus our attention on the main ideas, we let these differentiability assumptions hold throughout the paper.

Although problem \eqref{PBP} {{is more difficult to solve}} \cite{Au,DempeBook,D2,D3,LamparielloEtal2019,ZemkohoSetValPaper1}, it has only scantly been addressed in the literature, in comparison to its optimistic counterpart, which is at the centre of almost all the attention; see, e.g., the books \cite{DempeBook,DempeZemkohoBook,OutrataBook} and references therein.
However, problem \eqref{PBP} is more realistic from the practical modelling perspective, given that it assumes that there is no cooperation between the leader and the follower, in contrary to the situation in the optimistic version of the problem. For a more detailed analysis of the pessimistic and optimistic bilevel optimization problems and the related structural challenges and differences between them, interested readers are referred to the latter references.

In this paper, we consider the Karush-Kuhn-{{Tucker}} (KKT) reformulation
\begin{equation}\tag{KKT}\label{MPCC}
    \underset{x\in X}\min~\psi_{p}(x):=\underset{(y, u)\in \mathcal{D}(x)}\max~F(x,y)
\end{equation}
of problem \eqref{PBP}, where the set-valued mapping $\mathcal{D} :\mathbb{R}^{n} \rightrightarrows \mathbb{R}^{m}\times\mathbb{R}^{q}$ is given by
\begin{equation}\label{KKT system}
\mathcal{D}(x) := \left\{\left. (y, u)\in \mathbb{R}^{m+q} \right|\;\; \mathcal{L}(x,y,u)=0, \;\; u\geq 0, \;\, g(x,y)\leq0, \;\, u^\top g(x, y)=0\right\}
\end{equation}
with $\mathcal{L}(x,y,u):=\nabla_{y}L(x,y,u)$ for the Lagrangian function $L:(x,y,u)\rightarrow L(x,y,u):=f(x,y)+u^{T}g(x,y)$  associated to the parametric optimization problem describing the lower-level optimal solution set-valued mapping \eqref{S(x)}.
%
Of course, as it is the case for any KKT reformulation in bilevel optimization \cite{Au,DempeDutta2012,DempeZemkohoKKTRefNonsmooth,DempeZemkohoKKTRef}, we will assume throughout the paper that the lower-level problem \eqref{S(x)} is convex (i.e., the functions $f$ and $g_i$, $i=1, \ldots, q$ are convex w.r.t. $y$) and a constraint qualification (CQ) is satisfied. In particular, we assume that the following {{parametric}}
Slater CQ  is satisfied throughout this article (see, e.g.,  \cite{Au2,DempeDutta2012}):
\begin{equation}\tag{SCQ}\label{SCQ}
\forall x\in X,\;\; \exists y(x)\in\mathbb{R}^{m}: \;g_{i}(x,y(x))<0,\;\;\forall i=1, \ldots, q.
\end{equation}
%
%
Obviously, problem \eqref{MPCC} is a special class of the minmax optimization problem with parametric complementarity or equilibrium constraints in the {lower-level} problem. In the sequel, we will refer to it as the KKT reformulation of the pessimistic bilevel optimization problem \eqref{PBP}.\\[-2.5ex]

\begin{algorithm}[H]\label{algorithm 1}
\caption{Scholtes relaxation method for pessimistic bilevel optimization}
\begin{algorithmic}
\STATE \textbf{Step 0}: Choose ${{t^0>0}}$ and set  $k:=0$.
\STATE \textbf{Step 1}: Solve problem \eqref{RMPCCt} for $t:=t^k$.
\STATE \textbf{Step 2}: Select ${{0<t^{k+1}<t^k}}$, set $k:=k+1$, and go to Step 1.
\end{algorithmic}
\end{algorithm}
${}$\\[-7.1ex]

The KKT reformulation for the optimistic counterpart of problem \eqref{PBP} is one of the standard approaches in the process of solving the problem. In the context of the KKT reformulation of the optimistic bilevel program, one of the most popular class of solution techniques is represented by relaxation methods, which enables a mitigation of the difficulties  caused by the presence of the complementarity constraints present in the inner feasible set of problem \eqref{MPCC}. Various relaxation methods have been proposed in the literature (see, e.g., \cite{HoheiselEtAlComparison2013}) to address the KKT reformulation of the optimistic bilevel optimization or more precisely, for the corresponding more general mathematical program with complementary constraints (MPCC). The latter publication provides a comparison of these relaxations and shows that the Scholtes relaxation introduced in  \cite{Scholtes2001} is not only simpler, from its construction, but it is also superior in terms of its  numerical efficiency. Motivated by this, we introduce the following Scholtes relaxation for the pessimistic bilevel program \eqref{PBP} via its KKT reformulation \eqref{MPCC}:
\begin{equation}\tag{KKT$_t$}\label{RMPCCt}
\underset{x\in X}{\min}~\psi_{p}^{t}(x):=\underset{(y,u)\in\mathcal{D}
^{t}(x)}{\max}F(x,y),
\end{equation}
where, for all $t>0$,  $\mathcal{D}^{t}$ represents the Scholtes perturbation of the {{set-valued mapping  $\mathcal{D}$ in}} \eqref{KKT system}
\begin{equation}\label{Dt}
\mathcal{D}^{t}(x):=\left\{(y,u)\in\mathbb{R}^{m+q}\left\vert\,
\mathcal{L}(x,y,u)=0, \;\, u\geq0,\;\; g(x,y)\leq0, \;\, -u_{i}g_{i}(x,y)\leq t, \;\; i=1,...,q \right.\right\},
\end{equation}
which is assumed to be nonempty throughout the paper.

We then introduce Algorithm \ref{algorithm 1}, which can be seen as the Scholtes relaxation algorithm for the KKT reformulation \eqref{MPCC} of the pessimistic bilevel optimization problem \eqref{PBP}. Our primary goal in this paper is to provide a theoretical framework ensuring that the limit of a sequence of global/local optimal solutions  $(x^k)$  for  \eqref{RMPCCt} for $t:=t^k$, computed through Algorithm \ref{algorithm 1}, is a global/local optimal solution of problem \eqref{MPCC} as $t^k\downarrow 0$. Considering the fact that problems \eqref{RMPCCt} and \eqref{MPCC} are both nonconvex, it will typically be more realistic to aim at computing their stationary points. Hence, we also carefully address the case where the sequence of points computed in Step 1 of Algorithm \ref{algorithm 1} are only stationary points $(\zeta^k)$  of \eqref{RMPCCt} for $t:=t^k$. And therefore, we establish that the limit of this sequence is  a stationary point of \eqref{MPCC} as $t^k\downarrow 0$, under suitable assumptions.
As we study these questions, we also provide technical conditions ensuring that the problem solved in Step 1 (i.e., solving problem  \eqref{RMPCCt} directly or computing its stationary points) remains tractably solvable as we get close to a limit point of interest.
It is important to note that a global/local optimal solution of the KKT reformulation \eqref{MPCC}  is equivalent to a global/local optimal solution of  \eqref{PBP} under mild assumptions \cite{Au}.

%

Also note that only a very small number of papers have investigated solution algorithms to solve problem \eqref{PBP}. As to the few works that we are aware of, we can mention \cite{D2}, which provides necessary optimality conditions for various scenarios of the problem; the results there are extended to nonsmooth functions in \cite{DMZns2019}. Various papers, including \cite{LoridanMorgan1989,LoridanMorgan1996}, from the same authors and their collaborators, propose approximations and stability results based on perturbations in the value function constraint of the lower-level value function reformulation of problem \eqref{PBP},  in the case where the lower-level feasible set is independent from $x$. In \cite{WiesemannTsoukalasKleniatiRustem2013}, a semi-infinite programming-based algorithm is developed for a slightly general version of problem \eqref{PBP}, also in the case where the lower-level feasible set is independent from $x$. The article \cite{CervinkaMatonaha} proposes a two step process to compute a special class of equilibrium for \eqref{PBP}, which consists of solving the problem with an off-the-shelf solver while for fixed values of $x$, the parametric optimization problem defining $\varphi_p$ is solved with a certain BFO (Brute-Force Optimizer). Drawing inspiration from the optimistic bilevel optimization literature, the paper \cite{LamparielloEtal2019} proposes a standard-type approximation model for problem \eqref{PBP}, where $F(x, y)$ is minimized w.r.t. $(x, y)$ over a feasible set described in part by  the optimal solution set of the problem to minimize $F(x,y)$ w.r.t. $y$ subject to an approximation of the lower-level optimal solution set-valued mapping $S$ \eqref{S(x)}. And finally, as already mentioned above, \cite{Au} studies the relationship between problem \eqref{PBP} and \eqref{MPCC}. Clearly, we are not aware of any attempt to solve problem \eqref{PBP} from a perspective that is close to the one considered in this paper.

In summary, the main contributions of this paper are as follows:
\begin{enumerate}
    \item We introduce the KKT reformulation \eqref{MPCC} of the pessimistic bilevel optimization problem \eqref{PBP}  and the corresponding Sholtes relaxation problem \eqref{RMPCCt}, with relaxation parameter $t>0$, and study some of its basic structural properties; cf. Section \ref{Basic set-valued mapping properties}.
    \item We introduce a Scholtes relaxation method tailored to the KKT reformulation of the pessimistic bilevel program \eqref{PBP} and prove its convergence to local/global optimal solutions for problem \eqref{MPCC} under tractable assumptions; see Section \ref{Computing global and local optimal solutions} for the corresponding details.
    \item We construct a suitable framework to ensure the convergence of our Scholtes relaxation method to a C-stationary point tailored to problem \eqref{MPCC}; cf. Section \ref{Computing C-stationary points}.
    \item Numerical experiments are conducted to show how the Scholtes relaxation method introduced in this paper can be applied to compute stationary for problem \eqref{MPCC}; cf. Section \ref{Sec:Numerical}. 
\end{enumerate}

{{The remainder}} of the paper is organized as follows. In the next section, we introduce basic concepts from variational analysis that will be useful in the subsequent sections, and also conduct a preliminary analysis of \eqref{RMPCCt}. Section \ref{Computing global and local optimal solutions} develops a framework to ensure that {{a global (resp. local) optimal solution for  \eqref{MPCC} can  be obtained from a converging sequence of global (resp. local) optimal solutions of \eqref{RMPCCt}}}. {{In Section \ref{Computing C-stationary points}, we study the situation where Step 1 of Algorithm \ref{algorithm 1} instead computes a special class of {{stationary}} points for \eqref{RMPCCt} and show how the limit of the corresponding converging  sequence of points is a C-stationary point for \eqref{MPCC}}}.  As we develop the theoretical results in Sections \ref{Basic set-valued mapping properties}--\ref{Computing C-stationary points}, illustrative examples and discussions are provided to show that the required assumptions are tractable.  In Section \ref{Sec:Numerical}, we conduct some numerical experiments to assess the performance of our Scholtes relaxation algorithm and compare its behavior on different forms of the optimality conditions of \eqref{RMPCCt}.  Conclusions and some ideas for future investigations are provided in Section \ref{sec:conclusions}. 

\section{Basic mathematical tools and preliminary analysis}\label{Basic set-valued mapping properties}
In this section, we briefly present some basic  properties, mostly related to variational analysis, which will be used throughout the paper; for more detail on the main concepts, see  \cite{BS2,RTW}, for instance.
We start with some notation. First, {{we denote by $B(a,r)$ the open ball in $\mathbb{R}^{n}$  with center the point $a\in \mathbb{R}^{n}$ and radus $r>0$ and we use $\mathbb{B}_{n}$ to denote the unit ball centered at the origin of the space}}. Given two subsets $A$ and $B$ of $\mathbb{R}^{n}$, their Hausdorff
distance $d_{H}(A,B)$ is given as
\[
d_{H}(A,B):=\max~\{e(A,B), \;\, e(B,A)\},
\]
where the excess $e(A,B)$ of $A$ over $B$ is defined by the formula
\[
e(A,B):=\underset{x\in A}{\sup}~d(x,B)
\]
with $d(x,B):=\underset{y\in B}{\inf}~d(x,y)$ representing the distance from the point $x$ to the set $B$, while considering the usual conventions
 $e(\varnothing,B)=0$ and $e(A,\varnothing)=+\infty$ if
$A\neq\varnothing$. Furthermore, let
\[
\mathrm{dom}\,\Psi:=\left\{x\in\mathbb{R}^{n}|\;\;\Psi(x)\neq\varnothing\right\} \;\, \mbox{ and }\;\, \mbox{gph}\,\Psi :=\{(x,y)\in\mathbb{R}^{n+m}|\;\;y\in\Psi(x)\}
\]
denote the domain  and graph of the set-valued mapping $\Psi :\mathbb{R}^{n}\rightrightarrows\mathbb{R}^{m}$, respectively. 

\begin{definition}
A set-valued mapping $\Psi:\mathbb{R}^{n}\rightrightarrows\mathbb{R}^{m}$
is said to be
\begin{description}
\item[(i)] inner semicompact at $\bar{x}\in\mathrm{dom}\Psi$ if for any
sequence $(x_{k})_k$ such that $x_{k}\rightarrow\bar{x}$ there exists a sequence $(y_{k})_k$
with $y_{k}\in\Psi(x_{k})$ such that $(y_{k})_k$ admits a convergent
subsequence;
\item[(ii)]  inner semicontinuous at $(\bar{x},\bar
{y})\in\mbox{gph }\Psi$ if for any sequence $(x_{k})_{k}$ such that $x_k \rightarrow\bar{x}$,
there exists a sequence $(y_{k})_k$ such that $y_{k}\in\Psi(x_{k})$ and
$y_{k}\rightarrow\bar{y}$ or equivalently, if $d(\bar{y},\Psi(x))\rightarrow0$
whenever $x\rightarrow\bar{x}$;
\item[(iii)] lower semicontinuous in the sense of Hausdorff (i.e., \textit{H-lower semicontinuous}, for short) at $\bar{x}
\in\mathrm{dom}\Psi$  if, for every $\varepsilon
>0$, there exists a neighborhood $U$ of $\bar{x}$ such that for any $x\in U$,
\[
e(\Psi(\bar{x}),\,\Psi(x))<\varepsilon \;\, \mbox{ or equivalently } \;\, \Psi(\bar{x})\subset\Psi(x)+\varepsilon\mathbb{B}_{m};
\]
\item[(iv)] upper semicontinuous in the sense of Hausdorff (i.e., \textit{H-upper semicontinuous}, for short) at $\bar{x}\in\mathrm{dom}\Psi$
 if for every $\varepsilon>0$, there exists a
neighborhood $U$ of $\bar{x}$ such that for any $x\in U$,
\[
e(\Psi(x),\,\Psi(\bar{x}))<\varepsilon  \;\, \mbox{ or equivalently } \;\,\Psi(x)\subset\Psi(\bar{x})+\varepsilon\mathbb{B}_{m};
\]
\item[(v)]  upper semicontinuous in the sense of Berge (i.e., \textit{B-upper semicontinuous}, for short) at $\bar{x}$ if  for each open set $V$ such that $\Psi(\bar{x})\subset V$, there exists a neighbourhood $U$ of $\bar{x}$ such that $\Psi(U)\subset V$;
\item[(vi)] lower semicontinuous in the sense of Berge (i.e., \textit{B-lower semicontinuous}, for short) at $\bar{x}$ if for each open set $V$ such that $V\cap \Psi(\bar{x})\neq \emptyset$, there is a neighbourhood $U$ of $\bar{x}$ such that for all $x\in U$, $V\cap \Psi(x)\neq \emptyset$;
\item[(vii)] Aubin continuous (or Lipschitz-like) around
$(\bar{x},\bar{y})\in\mbox{gph }\Psi$ if there is a  constant $\tau>0$ and there exist a
neighborhood $U\times V$ of $(\bar{x},\bar{y})$ such that for any
$x,x^{\prime}\in U$,
\begin{equation}\label{LipschitzContSet}
e(\Psi(x)\cap V,\,\Psi(x^{\prime}))\leq\tau\left\Vert x-x^{\prime}\right\Vert;
\end{equation}
i.e., for any $y\in\Psi(x)\cap V$, there exists $y^{\prime}\in\Psi
(x^{\prime})$ such that
$
\left\Vert y-y^{\prime}\right\Vert \leq\tau\left\Vert x-x^{\prime}\right\Vert.
$
If inequality \eqref{LipschitzContSet} is satisfied with $V=\mathbb{R}^m$, $\Psi$ is said to be Lipschitz continuous around $\bar{x}$.
\end{description}
\end{definition}

There are various interconnections between these concepts. Let us summarize a few key ones relevant to this paper. First, we have the following relationships between the Hausdorff and Berge semicontinuity:
\begin{theorem}[see Lemma 2.2.3 in \cite{KlatteEtAlBook1982}]\label{Lemma223}
Consider the set-valued mapping $\Psi:\mathbb{R}^{n}\rightrightarrows\mathbb{R}^{m}$  and let $\bar{x} \in \mathrm{dom}\Psi$:
\begin{description}
    \item[(i)] If $\Psi$ is H-lower semicontiuous  at $\bar{x}$, then $\Psi$ is B-lower semicontinuous  at $\bar{x}$. The converse holds true if we assume that the set $\Psi(\bar{x})$ is compact.
     \item[(ii)] If $\Psi$ is B-upper semicontiuous at $\bar{x}$, then $\Psi$ is H-upper semicontinuous  at $\bar{x}$ and the converse holds when one assumes that the set $\mathrm{cl}\Psi(\bar{x})$ is compact.
\end{description}
\end{theorem}

Observe that   the H-lower semicontinuity of $\Psi:\mathbb{R}^{n}\rightrightarrows\mathbb{R}^{m} $ at $\bar{x}$ implies the inner semicontinuity  of $\Psi $ at $(\bar{x},\bar{y})$ for every $\bar{y}\in \Psi(\bar{x})$. The Lipschitz-likeness of $\Psi $ around $(\bar{x},\bar{y})$  clearly implies the inner semicontinuity of $\Psi $ at $(\bar{x},\bar{y})$, which obviously implies the inner
semicompactness of $\Psi $ at $\bar{x}$. Moreover, any nonempty set-valued mapping  that is uniformly bounded
around $\bar{x}$ is obviously inner semicompact at this point.
It is important to recall that the Lipschitz-likeness of a set-valued mapping $\Psi :\mathbb{R}^{n}\rightrightarrows\mathbb{R}^{m} $ around a point $(\bar{x},\bar{y})$ is ensured by the Mordukhovich criterion \cite{Mork} (see also Theorem 9.40 in \cite{RTW}); i.e., $\Psi $ is Lipschitz-like around $(\bar{x},\bar{y})$ if and only if
\[
D^*\Psi(\bar{x},\bar{y})(0)=\left\lbrace 0\right\rbrace,
\]
provided that the graph of $\Psi$ is closed. Here, $D^*\Psi(\bar{x},\bar{y}):\mathbb{R}^{m}\rightrightarrows\mathbb{R}^{n}$ denotes the coderivative of  $\Psi$ at  $(\bar{x},\bar{y})$, defined by
$w \in D^*\Psi(\bar{x},\bar{y})(v)$ if and only if $(w,-v) \in N_{\mathrm{gph}\Psi}(\bar{x},\bar{y})$
with $N_{\mathrm{gph}\Psi}(\bar{x},\bar{y})$ being the limiting normal cone to $\mbox{gph}\,\Psi$ at $(\bar{x},\bar{y})$; see relevant details in the latter references.

In the sequel, we will also use the concept of Painlev\'{e}--Kuratowski \emph{outer/upper} and \emph{inner/lower} limit for a set-valued mapping $\Psi :\mathbb{R}^n \rightrightarrows \mathbb{R}^m$ at a point $\bar{x} \in\mathrm{dom}\Psi$, which are respectively defined by
\[
\begin{array}{rcl}
\underset{x \longrightarrow \bar x}\limsup~\Psi(x) &:=& \left\{y\in \mathbb{R}^m |\;\;\exists x_k \rightarrow \bar x, \;\; ,y_k \rightarrow y \;\; \mbox{ with } \;y_k\-\in \Psi(x_k)\; \mbox{ for all }\; k  \right\},\\[2ex]
\underset{x \longrightarrow \bar x}\liminf~\Psi(x)&:=&\left\{y\in \mathbb{R}^m |\;\;\forall x_k \rightarrow \bar x, \;\, \exists y_k \rightarrow y \; \mbox{ with } \;y_k\-\in \Psi(x_k)\; \mbox{ for all }\; k  \right\}.
\end{array}
\]
Recall that if $ \underset{x \longrightarrow \bar x}\limsup~\Psi(x) = \underset{x \longrightarrow \bar x}\liminf~\Psi(x)$, then the Painlev\'{e}--Kuratowski limit is said to exist at $\bar x$ and can simply be denoted by $\underset{x \longrightarrow \bar x}\lim~\Psi(x)=\underset{x \longrightarrow \bar x}\limsup~\Psi(x) = \underset{x \longrightarrow \bar x}\liminf~\Psi(x)$.

Next, we collect some basic properties on the mappings
defined  in \eqref{KKT system} and \eqref{Dt}.
\begin{proposition}\label{lem} For any $x\in\mathbb{R}^{n}$, it holds that:
\begin{description}
\item[(i)] $\mathcal{D}(x)$ and $\mathcal{D}^{t}(x)$ are closed ($t>0$);

\item[(ii)] $\mathcal{D}^{t_{1}}(x)\subset\mathcal{D}^{t_{2}}(x)$ for any
$t_{2}>t_{1}>0$;

\item[(iii)] $\mathcal{D}(x)=$ $\underset{t>0}{\cap}\mathcal{D}^{t}(x)\ $ and so for all
$t>0$,
$
e(\mathcal{D}(x),\mathcal{D}^{t}(x))=0
$
and
$\psi_p(x)\leq\psi^t_p(x)$;
\item[(iv)] $\mathcal{D}(x)=$ $\underset{t\downarrow0}{\lim}\mathcal{D}^{t}(x)$ in
the Painlev\'{e}-Kuratowski sense.
\end{description}
\end{proposition}
\begin{proof}
For (i), note by continuity of the functions describing the sets $\mathcal{D}(x)$ and $\mathcal{D}^{t}(x)$, they are obviously closed.
Assertions (ii) is also obvious and so is  assertion (iii) given that $\mathcal{D}
(x)\subset\mathcal{D}^{t}(x)$ for all $t>0$.  Let us prove the Painlev\'{e}
-Kuratowski convergence (assertion (iv)); i.e.,
\[
\underset{t\downarrow0}{\lim\inf}\,\mathcal{D}^{t}(x)=\mathcal{D}(x)=\text{
}\underset{t\downarrow0}{\lim\sup}\,\mathcal{D}^{t}(x).
\]
Note that from assertion (iii), it holds that  $\mathcal{D}(x)\subset$ $\underset{t\rightarrow0^{+}}{\lim\inf
}\,\mathcal{D}^{t}(x)\subset$ $\underset{t\rightarrow0^{+}}{\lim\sup}\, \mathcal{D}^{t}(x)$.

Conversely, taking $(y,u)\in$ $\underset{t\downarrow0}{\lim\sup}\,\mathcal{D}^{t}(x)$, it follows that
for a sequence $(t_{k})_k$ {{converging}} to $0$, the point $(y,u)$ is a cluster point of a sequence
in $\mathcal{D}^{t_{k}}(x)$; i.e., there exists a sequence $(y^{k},u^{k})_{k}$ with
$(y^{k},u^{k})\in\mathcal{D}^{t_{k}}(x)$ which converges
(up to a subsequence) to $(y,u)$. Thus, for any $k$, we have
\[
\begin{array}
[c]{l}
\mathcal{L}(x,y^{k},u^{k})=0, \;\; u_{i}^{k}\geq0,\text{ }g_{i}(x,y^{k})\leq 0, \;\; -u_{i}^{k}g_{i}(x,y^{k})\leq t_{k}, \;\; i=1,...,q.
\end{array}
\]
Applying the limit to this system, we get $(y,u)\in\mathcal{D}(x)$. Hence,  $\underset
{t\rightarrow0^{+}}{\lim\sup}\,\mathcal{D}^{t}(x)\subset\,\mathcal{D}(x).$
\hfill\qed
\end{proof}

Next, we show that for a fixed point $x$, a sequence of objective function values of problem \eqref{RMPCCt} for $t:=t_k$ can converge to that of problem \eqref{MPCC} as the sequence $(t_k)_k$ converges to $0$ as $k\rightarrow\infty$.
\begin{proposition}\label{FirstProp}
Let $t\mapsto\psi_{p}^{t}(x)$ be upper semicontinuous at $0^+$, for any $x\in \mathbb{R}^n$ and
let $t_{k}\downarrow0$. Then for any $x\in\mathbb{R}^{n}$, we have
$\psi_{p}^{t_{k}}(x)\rightarrow\psi_{p}(x)$ as $k\rightarrow\infty$.
\end{proposition}
\begin{proof}
It is obvious that for any $x\in\mathbb{R}^{n},$ the sequence $(\psi
_{p}^{t_{k}}(x))$ converges for any sequence $(t_{k})_k$ {{converging}} to $0$.
Indeed, since $t_{k+1}\leq t_{k}$, we have $\mathcal{D}^{t_{k+1}}
(x)\subset\mathcal{D}^{t_{k}}(x)$ so that $\psi_{p}^{t_{k+1}}(x)\leq\psi
_{p}^{t_{k}}(x)$ for any $k$ and since the sequence is bounded from below by
$\psi_{p}(x)$, it converges. Hence,
\begin{equation}\tag*{\mbox{\qed}}
\psi_{p}(x)\leq\underset{k\rightarrow\infty}{\lim}\psi_{p}^{t_{k}}(x)=\lim \sup\psi_{p}^{t_{k}}(x)\leq\psi_{p}(x) \;\; \mbox{ for every }\;\; x\in \mathbb{R}^n.   
\end{equation}
\end{proof}
We conclude this section with an illustrative example showing how the feasible set and objective function of problem \eqref{MPCC} change under the Scholtes-type relaxation.
\begin{example}\label{ex4} Consider an example of the pessimistic bilevel problem \eqref{PBP} with
\begin{equation}\label{ex2pb}
F(x,y):=x+y, \;\; X:=[-1, \, 1], \;\; f(x,y):=xy, \;\mbox{ and }\; K(x):=[0, \, 1].
\end{equation}
We can easily check that
\[
S(x) = \left\{
\begin{array}[c]{ll}
\left\{  0\right\} & \;\text{  if }\; 0<x\leq1,\\
\left\{  1\right\} & \;\text{ if }\; -1\leq x<0,\\
\left[  0,1\right] & \;\text{ if }\; x=0,
\end{array}
\right.\;\; \mbox{ and }\;\; \varphi_{p}(x)=\left\{
\begin{array}[c]{ll}
x & \text{  if }\; 0<x\leq1,\\
x+1 & \text{ if }\; -1\leq x\leq0.
\end{array}
\right.
\]
Clearly, $\bar{x}=-1$ is the global optimal solution for problem \eqref{ex2pb}. Taking into account that the Lagrangian function $L(x,y,u)=xy+u_{1}g_{1}(x,y)+u_{2}g_{2}(x,y)\;\;\text{with}\;\; g_1(x,y):=-y\;\; \text{and}\;\; g_2(x,y):=y-1$, the conditions
\begin{equation*}
    \mathcal{L}(x,y,u)=0, \;\; u_i\geq0,\;\;\;\text{and}\;\;\; g_i(x,y)\leq0 \;\;\mbox{ for }\;\, i=1,\; 2
\end{equation*}
    are equivalent to
\begin{equation}\label{exp1}
     u_1\geq0,\;\; u_2=u_1-x\geq0, \; \;\;\text{and}\;\;\; 0\leq y\leq 1.
\end{equation}
\begin{figure}
    \centering
    \includegraphics[width = .68\linewidth]{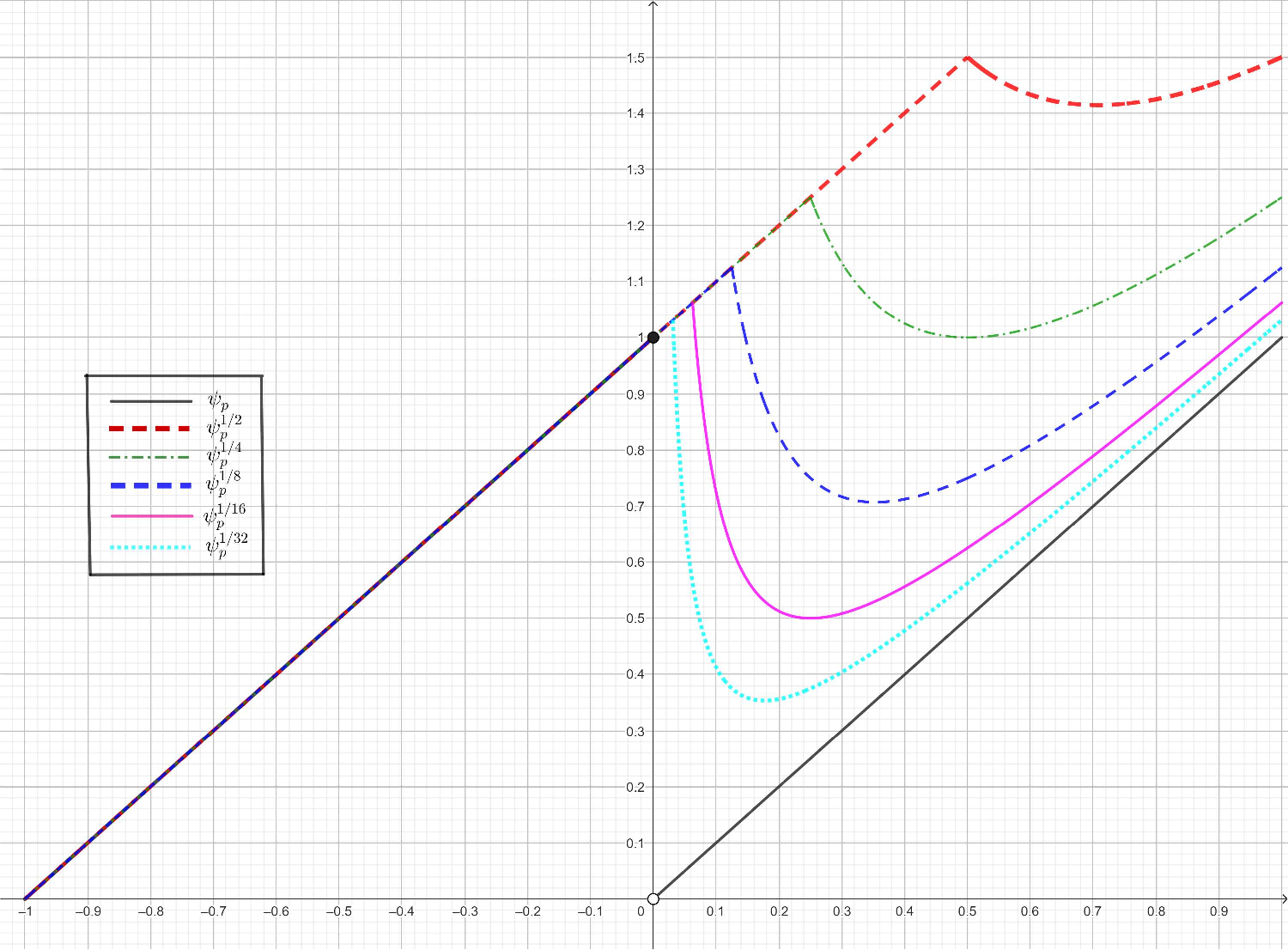}
    \caption{Graphical representation of the functions $\psi_p$ \eqref{MPCC} and $\psi^t_p$ \eqref{RMPCCt}, for the scenarios $t=\frac{1}{2},  \frac{1}{4}, \frac{1}{8}, \frac{1}{16}$, and $\frac{1}{32}$, in the context of the problem in Example \ref{ex4}.}
    \label{fig:my_label_2}
\end{figure}
Thus, the corresponding version of problem \eqref{MPCC} can be obtained with
\[
\mathcal{D}(0)=\left[0, \, 1\right]\times\left\{(0,0)\right\} \;\, \mbox{ and }\;\,
\mathcal{D}(x)=\left\{
\begin{array}{lll}
\left\{(0,x,0)\right\}         & \;\text{if} & \;0<x \leq 1,\\[1ex]
\left\{(1,\, 0,\, -x)\right\}  & \;\text{if} & \;-1\leq x<0,
\end{array}
\right.
\]
so that  $\psi_p=\varphi_p$,  which is continuous at any $x\neq 0$.
Moreover, since for $(y,u)$ satisfying \eqref{exp1}, one has
\begin{equation*}
\left( -u_{i}g_{i}(x,y)\leq t \mbox{ for } i=1,\,2\right)\;\; \Longleftrightarrow \;\; \left( u_1y\leq t\;\; \text{and}\;\;  (u_2-x)(y-1)\geq -t\right)
\end{equation*}
so that the regularized version \eqref{RMPCCt} of the problem can be written
with $\mathcal{D}^{t}(x)$  for $0<t\leq x\leq1$ as
\begin{equation*}
\mathcal{D}^{t}(x)=\left\{\left(y, \,(u_{1}, \,u_{1}-x)\right)\;\;\; \left\vert
\begin{array}[c]{c}
\left(x\leq u_{1}<t+x\;\;\wedge\;\; 0\leq y\leq\dfrac{t}{u_{1}}\right)\\
\vee\\
\left(t+x\leq u_{1}\leq \bar{u}^t(x)\;\;\wedge\;\;-\dfrac{t}{u_{1}-x}+1\leq y\leq
\dfrac{t}{u_{1}}\right)
\end{array}
\right.  \right\},
\end{equation*}
where $\bar{u}^{t}(x)$ is the root of the equation  $1-\dfrac{t}{z-x}=\dfrac{t}{z}$ (for fixed $t$ and $x$) that satisfies $\bar{u}^t(x)>\max(0,x)$, namely, $\bar{u}^t(x)=\dfrac{2t+x+\sqrt{4t^{2}+x^{2}}}{2}$.
And for $0\leq x < t\leq1$,
\begin{equation*}
\mathcal{D}^{t}(x)=\left\{(y,(u_{1},\;\, u_{1}-x))\;\;\; \left\vert
\begin{array}[c]{c}
\left(x\leq u_{1}\leq t\;\; \wedge\;\; 0\leq y\leq1\right)\\[1ex]
\vee\\
\left(t<u_{1}\leq t+x\;\;\wedge\;\; 0\leq y\leq\dfrac{t}{u_{1}}\right)\\[1ex]
\vee\\
\left(t+x\leq u_{1}\leq \bar{u}^t(x)\;\;\wedge\;\;-\dfrac{t}{u_{1}-x}+1\leq y\leq \dfrac{t}{u_{1}}\right)
\end{array}
\right.  \right\}. \label{dt}
\end{equation*}
On the other hand for $-1\leq x\leq-t<0$, we obtain
\[
\mathcal{D}^{t}(x)=\left\{\left(y,\, (u_{1}, \,u_{1}-x)\right)\;\;\;\left\vert
\begin{array}[c]{c}
\left(0\leq u_{1}<t\;\;\wedge\;\;-\dfrac{t}{u_{1}-x}+1\leq y\leq1\right)\\[2ex]
\vee\\
\left(t\leq u_{1}\leq \bar{u}^t(x)\;\;\wedge\;\;-\dfrac{t}{u_{1}-x}+1\leq y\leq\dfrac
{t}{u_{1}}\right)
\end{array}
\right.  \right\}
\]
and for $-1\leq-t\leq x<0$, we get
\[
\mathcal{D}^{t}(x)=\left\{(y, \, (u_{1}, \,u_{1}-x))\;\;\; \left\vert
\begin{array}[c]{c}
\left(0\leq u_{1}\leq t+x\wedge0\leq y\leq1\right)\\[2ex]
\vee\\
\left(t+x<u_{1}\leq t\wedge-\dfrac{t}{u_{1}-x}+1\leq y\leq1\right)\\[2ex]
\vee\\
\left(t\leq u_{1}\leq \bar{u}^t(x)\wedge-\dfrac{t}{u_{1}-x}+1\leq y\leq\dfrac
{t}{u_{1}}\right)
\end{array}
\right.  \right\}.
\]
And subsequently, we have
\begin{equation*}
\psi^{t}_p(x) = \left\{
\begin{array}[l]{ll}
x + \dfrac{t}{x} & \;\mbox{if }\;\,t\leq x,\\[1ex]
x + 1 & \;\mbox{if }\;\, x<t.
\end{array}
\right.
\end{equation*}
Observe that for each $t>0$, the regularized function $\psi^{t}_p$ is continuous everywhere on the feasible set, in contrary to the function  $\psi_p$  which is everywhere except at $0$; see the graphically illustrations in Figure \ref{fig:my_label_2}. \hfill \qed
\end{example}

\section{Computing global and local optimal solutions}\label{Computing global and local optimal solutions}
%
%
In this section, we assume that in Step 1 of Algorithm \ref{algorithm 1}, we compute a global or local optimal solution of problem \eqref{RMPCCt} for $t:=t_k$ in the process of solving problem \eqref{MPCC}. Recall that a point $\bar{x}\in X$ is a local optimal solution for \eqref{PBP}
(resp. \eqref{MPCC}) if there exists a neighborhood $U$ of $\bar{x}$ such that condition
\begin{equation}\label{OptimalSolDef}
\forall x\in X\cap U: \;\varphi_{p}(\bar{x})\leq\varphi_{p}(x)\text{ \ (resp.
}\psi_{p}(\bar{x})\leq\psi_{p}(x)\text{)}
\end{equation}
is satisfied. Similarly, $\bar{x}\in X$ will be said to be a global optimal solution for \eqref{PBP} (resp. \eqref{MPCC})  if condition \eqref{OptimalSolDef} holds with $U=\mathbb{R}^n$.
%
%
%
%
%
We now state the convergence of Algorithm \ref{algorithm 1} when  \eqref{RMPCCt} is solved globally as $t\downarrow 0$.
\begin{theorem}\label{sg}
Let the function $x \mapsto \psi_{p}(x)$ be lower semicontinuous at $\bar x$ and $t\mapsto\psi_{p}^{t}(x)$ be upper
semicontinuous at $0^+$ for all $x\in \mathbb{R}^n$. Furthermore, let $t_{k}\downarrow0 $
and {{$(x^{k})_k$ be a sequence such that the point $x^{k}$ is a global optimal
solution of \eqref{RMPCCt} for $t:=t_k$. If the sequence $(x^{k})_k$ admits a subsequence (with the same notation) converging to  the point $\bar{x}$ as $k\rightarrow\infty$}}, then the point $\bar{x}$ is a global optimal solution of  \eqref{MPCC}.
\end{theorem}
\begin{proof}
Based on the definitions of $\psi_{p}^{t}$ and $\psi_{p}$, as well as the fact that $x^{k}$ is a
global optimal solution of problem \eqref{RMPCCt} for $t:=t_k$, we have
\[
\psi_{p}(x^{k})\leq\psi_{p}^{t_{k}}(x^{k})\leq\psi_{p}^{t_{k}}(x)\;\, \mbox{ for all }\;\, x\in X
\]
with $x^{k}\in X$ and the first inequality resulting from Proposition \ref{lem}(iii). Since the subsequence $x^{k}\rightarrow\bar{x}$ as $k\rightarrow\infty$, $\bar{x}\in X$ ($X$ being closed by continuity of $G$). Now, let $x\in X$. Given that $\psi_{p}$ is lower semicontinuous at $\bar{x}$ and $t\mapsto \psi_{p}^{t}(x)$ is upper semicontinuous at $0$, for all $x\in \mathbb{R}^n$, we have the string of inequalities
\begin{equation}\tag*{\mbox{\qed}}
\psi_{p}(\bar{x})\leq\text{ }\underset{k\rightarrow\infty}{\lim\inf}\,\psi
_{p}(x^{k})\leq\text{ }\underset{k\rightarrow\infty}{\lim\inf}\,\psi_{p}^{t_{k}
}(x)\leq\text{ }\underset{k\rightarrow\infty}{\lim\sup}\,\psi_{p}^{t_{k}}
(x)\leq\psi_{p}(x).
\end{equation}
\end{proof}

Considering the fact that problems  \eqref{MPCC} and \eqref{RMPCCt} are both nonconvex, it is more likely that in practice, a scheme to solve either problem would only compute local optimal solutions. Hence, we next provide a result where iteratively solving  \eqref{RMPCCt}  can ensure that a local optimal solution for  \eqref{MPCC} is obtained.

\begin{theorem}\label{Th32Here}
	Let the function $x \mapsto \psi_{p}(x)$ be lower semicontinuous
	$\bar{x}$, and $\bar{r}>0$ be such that $(t,u)\mapsto\psi_{p}^{t}(u)$ is upper
	semicontinuous at $(0,x)$ for any $x\in B(\bar{x},\bar{r})$. Furthermore, let
	$t_{k}\downarrow0$, $r^{k}>0$, and $(x^{k})_k$ be a sequence of optimal solutions of
	problem \eqref{RMPCCt}  for $t:=t^k$ in $X\cap B(x^{k}, r^{k})$; i.e.,
	\begin{equation*}
	\psi_{p}^{t_{k}}(x^{k})\leq\psi_{p}^{t_{k}}(x)\text{ \ \ }\forall x\in
	X\cap B(x^{k},r^{k}).\label{sl}
	\end{equation*}
{Then the point $\bar{x}$ is a local optimal solution of problem \eqref{MPCC}, if additionally, it holds that $ \bar {x}$ is a cluster point of the sequence $(x^{k})_k$   and $\underset{k\rightarrow\infty}{\lim\inf}\,r^{k}>0$.}
\end{theorem}
\begin{proof}
	{{Take $0<\bar{r}<\underset{k\rightarrow\infty}{\text{ }\lim\inf}\,r^{k}$ such that for
	$k$ large enough, $r^{k}>\bar{r}$. Let $x\in X\cap B\left(\bar{x},\,\dfrac{\bar{r}
	}{2}\right)$. Since there exists a subsequence of  $(x_k)$ (denoted for simplicity also by $(x_k)$) such that $x^{k}\rightarrow\bar{x}$ as $k\rightarrow+\infty,$ $x^{k}\in	B\left(\bar{x},\dfrac{\bar{r}}{2}\right)$ for any $k$ large enough,  $x\in
B(x^{k}, \,r^{k})$ given that
\[
	\left\Vert x-x^{k}\right\Vert \leq\left\Vert x-\bar{x}\right\Vert +\left\Vert
	\bar{x}-x^{k}\right\Vert <\bar{r}<r^{k}.
	\]
	Thus, $x\in X\cap B(x^{k},r^{k})$ and
	$	\psi_{p}^{t_{k}}(x^{k})\leq\psi_{p}^{t_{k}}(x)
	$
	for any $k$ sufficiently large. We can then conclude the proof by proceeding as in the proof of Theorem \ref{sg}}}. \hfill \qed
\end{proof}
{{Observe that in this theorem, the limit--based assumption that  $\underset{k\rightarrow\infty}{\lim\inf}\,r^{k}>0$ can be replaced by the chain of inequalities  $r^{k}$ $>$ $\bar{r}$ $>0$ (for $k$ sufficiently large), which is much more simpler to check.}} 

In Proposition \ref{FirstProp}, Theorem \ref{sg}, and Theorem  \ref{Th32Here}, some continuity assumptions are imposed on the function $\psi^t_p$ ($t\geq 0$). More precisely, in these results, we require the following assumptions:
\begin{description}
\item[(A1)] $x \mapsto \psi_{p}(x)$ is a lower semicontinuous function at
	$\bar{x}$ (see Theorem \ref{sg} and Theorem  \ref{Th32Here});
\item[(A2)]  $t\mapsto\psi_{p}^{t}(x)$ is upper semicontinuous at $0^+$ for all $x\in \mathbb{R}^n$ (see Proposition \ref{FirstProp} and Theorem \ref{sg});
\item[(A3)] $(t,u)\mapsto\psi_{p}^{t}(u)$ is an upper
	semicontinuous function at $(0,x)$ for any $x\in B(\bar{x},\bar{r})$ (see Theorem  \ref{Th32Here}).
\end{description}
Moreover, in Theorem \ref{sg} and Theorem  \ref{Th32Here}, we require the existence of a sequence $(x^{k})_k$ such that the point $x^{k}$ is an optimal
solution of problem \eqref{RMPCCt} for $t:=t_k$. For a fixed value of $t>0$, problem \eqref{RMPCCt} has an optimal solution if the set $X$ is compact and it holds that
\begin{description}
\item[(A4)] the function $x \mapsto \psi_{p}^{t}(x)$ is lower semicontinuous on $X$.
\end{description}

However, in practice, to solve problem \eqref{RMPCCt}, for  a fixed value of $t>0$, as a nonsmooth minimization problem, we might need a stronger assumption on $\psi_{p}^{t}$. In particular, if we want to extend the gradient descent method to Step 1 of Algorithm \ref{algorithm 1}, for example, we will need the following assumption to calculate the generalized gradient (or {\em subdifferential}) of this function in the sense of Clarke (\cite{Clarke1990}):
\begin{description}
\item[(A5)] the function $x \mapsto \psi_{p}^{t}(x)$ is Lipschitz continuous around the point $x^t$.
\end{description}
Overall, this means that assumptions (A1)--(A5) are important for a nice interaction between problems \eqref{MPCC} and \eqref{RMPCCt}, as well as  for practically solving the latter problem (with certain types of methods), for a fixed value of $t>0$, as required in Step 1 of Algorithm \ref{algorithm 1}. Based on \cite[Theorem 4.2.3]{KlatteEtAlBook1982} and \cite[Theorem 5.3]{B1S}, we next summarize some key results ensuring that these assumptions are satisfied. To proceed, let
\[
\mathcal{S}_{p}(x):=\left\{  (y,u)\in\mathcal{D}(x)\left\vert \;F(x,y)\geq
\psi_{p}(x)\right.  \right\} \; \mbox{ and } \;
\mathcal{S}_{p}^{t}(x):=\left\{  (y,u)\in\mathcal{D}^{t}(x)\left\vert
\;F(x,y)\geq\psi_{p}^{t}(x)\right.  \right\}
\]
describe the optimal solution set-valued mappings of the parametric optimization problems associated to the optimal value functions $\psi_p$ and $\psi^t_p$, respectively.
\begin{theorem}\label{ExistencePt} The following assertions are satisfied:
    \begin{description}
        \item[(i)] The function $x \mapsto \psi_{p}(x)$ is lower semicontinuous at the point $\bar x$ if the set-valued mapping $\mathcal{D}(\cdot)$ is B-lower semicontinuous  at the point $\bar{x}$.
        \item[(ii)] For any $x\in \mathbb{R}^n$, the function $t\rightarrow \psi_{p}^{t}(x)$ is upper semicontinuous at $\Bar{t}$ if the set-valued mapping $t\rightrightarrows \mathcal{D}^{t}(x)$ is H-upper semicontinuous at  $\Bar{t}$ and  the set $\mathcal{D}^{\Bar{t}}(x)$ is compact.
\item[(iii)] The function $(t, x)\rightarrow\psi_{p}^{t}(x)$ is lower (resp. upper) semicontinuous at the point $(\Bar{t},\Bar{x})$ if the set-valued mapping $(t,x)\rightrightarrows \mathcal{D}^{t}(x)$ is B-lower (resp. H-upper) semicontinuous at  $(\Bar{t},\Bar{x}) $ (resp. and $\mathcal{D}^{\Bar{t}}(\Bar{x})$ is compact).
        \item[(iv)] For $t > 0$, the function $x \mapsto \psi_{p}^{t}(x)$ is lower semicontinuous at the point $\bar x$ if the set-valued mapping $\mathcal{D}^{t}(\cdot)$ is B-lower semicontinuous $\bar{x}$.
     \item[(v)]For $t > 0$, the function $x \mapsto \psi_{p}^{t}(x)$ is Lipschitz continuous around $\bar x$ if condition (a) or (b) below holds:
     \begin{description}
         \item[(a)] $\mathcal{S}_{p}^{t}(\cdot)$  is inner semicontinuous at $(\bar{x},\bar{y},\bar{u})\in \mbox{gph}\, \Psi_{p}^{t}$ and $\mathcal{D}^{t}(\cdot)$ is Lipschitz-like  around $(\bar{x},\bar{y},\bar{u})$.\\[-2ex]
         \item[(b)] $\mathcal{S}_{p}^{t}(\cdot)$  is inner semicompact at $\bar{x}$ and $\mathcal{D}^{t}(\cdot)$ is Lipschitz-like  around $(\bar{x}, y, u)$ for all $(y, u) \in \mathcal{S}_{p}^{t}(\bar{x})$.
     \end{description}
    \end{description}
\end{theorem}
Clearly, in this order, the assumptions in Theorem \ref{ExistencePt}(i), (ii), (iii), (iv), and (v) ensure the fulfillment of assumptions (A1), (A2), (A3), (A4), and (A5), respectively, at appropriately chosen points. Considering the importance of Step 1 in a practical implementation of Algorithm \ref{algorithm 1}, the framework for guaranteeing the existence of optimal solutions and/or for solving problem \eqref{RMPCCt}, for a fixed value of $t>0$, as a Lipschitz optimization problem (see, e.g., \cite{BS2,RTW} and references therein), is crucial. Hence, for the remainder of this section, we dive deeper into the analysis of the conditions ensuring that suitable versions of assumptions (A4) and (A5) can hold; a similar analysis can be done for (A1)--(A3). In particular, mainly for illustrative purposes, we focus our attention on the requirements for the assumptions needed in Theorem \ref{ExistencePt}(iv) and Theorem \ref{ExistencePt}(v)(a) to hold; namely, for a fixed value of $t>0$, these assumptions are precisely:
\begin{description}
\item[(B1)] set-valued mapping $\mathcal{D}^{t}(\cdot)$ is B-lower semicontinuous $\bar{x}$;
\item[(B2)] the set-valued mapping $\mathcal{D}^{t}(\cdot)$ is Lipschitz-like  around $(\bar{x},\bar{y},\bar{u})$;
\item[(B3)] the set-valued mapping $\mathcal{S}_{p}^{t}(\cdot)$  is inner semicontinuous at $(\bar{x},\bar{y},\bar{u})\in \mbox{gph} \Psi_{p}^{t}$.
\end{description}
 An important  question that we are interested in is to know what assumptions can be imposed on the data of our problem \eqref{MPCC} to ensure the satisfaction of these properties in relevant neighborhoods as $t\downarrow 0$? Some answers to this question are provided in the next result. To proceed, let us consider the  set-valued mappings $\mathcal{D}(\cdot,x): t \rightrightarrows
\mathcal{D}^{t}\mathcal{(}x\mathcal{)}$ for
$x\in\mathbb{R}^{n}$ and $\mathcal{D}(\cdot,\cdot\mathcal{)}:(t,x) \rightrightarrows
\mathcal{D(}t,x\mathcal{)}:=\mathcal{D}^{t}\mathcal{(}x\mathcal{)}$, where
$\mathcal{D}^{t}\mathcal{(}x\mathcal{)}$ is defined in (\ref{Dt}) for $t>0$
and $\mathcal{D}^{0}(x)=\mathcal{D}(x)$ is given in equation \eqref{KKT system}.
\begin{theorem}\label{prop}
\label{Continuity of set-valued maps}
The following statements hold:
\begin{description}
\item[(i)] Let the set-valued mapping $\mathcal{D(\cdot)}$ be H-lower semicontinuous at $\bar{x}$ and
$\mathcal{D}(\cdot,\cdot)$ be H-upper semicontinuous at $(0^+,\bar{x})$. Then, there exists a neighborhood $U$ of the point $\bar{x}$ such that for $t\downarrow0$, the set-valued mapping $\mathcal{D}^{t}(\cdot)$ is H-lower semicontinuous at all $x\in U$.
\item[(ii)] Let $\mathcal{D(\cdot)}$ be
Lipschitz-like around $(\bar{x},\bar{y},\bar{u})\in \mbox{gph}\mathcal{D}$ and $\mathcal{D}(\cdot, x)$
be H-upper semicontinuous at $0^{+}$ for any $x$ around $\bar{x}$. Then, there
exists a neighborhood $U\times V$ of $(\bar{x},\bar{y},\bar{u})$ such that
for $t\downarrow0$, $\mathcal{D}^{t}$ is Lipschitz-like around all points
$(x,y,u)$ with $x\in U$ and $(y,u)\in V\cap\mathcal{D}^{t}(x)$.
\item[(iii)] {{Let  $\mathcal{S}_{p}(\cdot)$  be inner semicontinuous at $(\bar{x},\bar{y},\bar{u})\in \mbox{gph}\mathcal{S}_{p}$ and  $t \rightrightarrows \mathcal{S}_{p}^{t}(x)$ be H-lower semicontinuous at $0^{+}$ for any $x$ close to $\bar{x}$}}. Then, there exists a neighborhood $U\times V$ of
$(\bar{x},\bar{y},\bar{u})$ such that for $t\downarrow0$, $\mathcal{S}
_{p}^{t}\mathcal{(\cdot)}$ is inner semicontinuous at all points $(x,y,u)$
with $x\in U$ and $(y,u)\in V\cap\mathcal{S}_{p}^{t}(x)$.
\end{description}
\end{theorem}

\begin{proof}
For (i), let $\varepsilon>0.$ Since $\mathcal{D}(\cdot,\cdot\mathcal{)}$ is
H-upper semicontinuous at $(0^+,\bar{x})$, there exist $\alpha>0$ and $r_{1}>0$ such that
for any $t\in(0,\alpha)$ and $x\in B(\bar{x},r_{1})$, one has
\[
\mathcal{D}^{t}(x)\subset D(\bar{x})+
\frac{\varepsilon}{2}\mathbb{B}_{m+q}
\]
and since $\mathcal{D(\cdot)}$ is H-lower semicontinuous at $\bar{x}$, there
exists $r_{2}>0$ such that for any $x^{\prime}$ in $B(\bar{x},r_{2})$,
\[
\mathcal{D}(\bar{x}) \;\; \subset \;\;  \mathcal{D}(x^\prime) + \frac{\varepsilon}{2}\mathbb{B}_{m+q}
 \;\; \subset\;\; \mathcal{D}^t(x^\prime) + \frac{\varepsilon}{2}\mathbb{B}_{m+q} \;\; \mbox{ for all } \; t>0.
\]
Hence, there exists a neighborhood $U$ of $\bar{x}$ such that for any
$x,x^{\prime}\in U$ and for any $t\in(0,\alpha)$,
\[
\mathcal{D}^{t}(x) \;\; \subset\;\; \mathcal{D}^{t}(x^{\prime})\; +\; \varepsilon\mathbb{B}_{m+q}.
\]
This means that $\mathcal{D}^{t}(\cdot)$ is H-lower semicontinuous at any $x\in U$.

For (ii), note that as $\mathcal{D}$ is Lipschitz-like around $(\bar{x},\bar{y},\bar
{u})\in\mbox{gph}\,\mathcal{D}(\cdot)$, there are strictly positive numbers $\tau$, $\delta_{1}$, and $r>0$ such that for
any $x,x^{\prime}\in B(\bar{x},\delta_{1})$, we have
\[
\mathcal{D}(x)\cap B((\bar{y},\bar{u}),r)\subset\mathcal{D}(x^{\prime}
)+\tau\left\Vert x-x^{\prime}\right\Vert \mathbb{B}_{m+q}
\]
and so
\begin{equation}
\mathcal{D}(x)\cap B((\bar{y},\bar{u}), \,r)\subset\mathcal{D}^{t}(x^{\prime
})+\tau\left\Vert x-x^{\prime}\right\Vert \mathbb{B}_{m+q}\label{Lipb}
\end{equation}
for all $t>0$. On the other hand, as $t \rightrightarrows \mathcal{D}^{t}(x)$ is
H-upper semicontinuous at $0^{+}$ for all $x$ around $\bar{x}$, then there exists $\delta_{2}>0$ so that  for any $x\in B(\bar{x},\delta_{2})$  and
$\varepsilon >0 $ there is a
constant $\alpha>0$  such that
\begin{equation}\label{us}
\mathcal{D}^{t}(x)\subset\mathcal{D}(x)+\varepsilon\mathbb{B}_{m+q} \;\, \mbox{ for all }\;\,  t\in(0,\alpha).
\end{equation}
{{Take $U := B\left(\bar{x},\,\dfrac{\delta}{2}\right)$ with $\delta=\min(\delta_{1},\delta_{2})$ and $V:=B\left((\bar{y},\bar{u}),\dfrac{r}{3}\right).$  Fix $x\in U $  and
$(y,u)\in V\cap\mathcal{D}^{t}(x).$ Let $\tilde U$ be a
neighborhood of $x$ such that $\tilde U\subset B(\bar{x},\delta)$ and $w,w^{\prime
}\in \tilde U$ with $w\neq w^{\prime}$. Given $0 <\varepsilon < \min(\dfrac{r}{3} ,\left\Vert w-w^{\prime}\right\Vert
)$, there exists some
$\alpha>0$ such that (\ref{us}) holds.  Taking $t\in(0,\alpha)$ as well as $(z,v)\in\mathcal{D}^{t}(w)\cap
B\left((y,u),\dfrac{r}{3}\right)$, one can pick}}
$(\hat{y},\hat{u})\in\mathcal{D}(w)$ such that
\[
\left\Vert (z,v)-(\hat{y},\hat{u})\right\Vert <\varepsilon,
\]
so that $(\hat{y},\hat{u})\in\mathcal{D}(w)\cap B((\bar{y},\bar{u}),r).$ Now,
from (\ref{Lipb}), it holds that
\[
(\hat{y},\hat{u})\in\mathcal{D}(w)\cap B((\bar{y},\bar{u}),r)\subset
\mathcal{D}^{t}(w^{\prime})+\tau\left\Vert w-w^{\prime}\right\Vert
\mathbb{B}_{m+q}.
\]
Therefore, there exists $\left(z^{\prime}, \;\, v^{\prime})\in\mathcal{D}^{t}(w^{\prime}\right)$ such
that the inequality
\[
\left\Vert (\hat{y},\hat{u})-(z^{\prime},v^{\prime})\right\Vert \leq
\tau\left\Vert w-w^{\prime}\right\Vert
\]
is satisfied. Hence, we have
\begin{align*}
\left\Vert (z,v)-(z^{\prime},v^{\prime})\right\Vert  & \leq\left\Vert
(z,v)-(\hat{y},\hat{u})\right\Vert +\left\Vert (\hat{y},\hat{u})-(z^{\prime
},v^{\prime})\right\Vert \\
& <\varepsilon+\tau\left\Vert w-w^{\prime}\right\Vert <(1+\tau)\left\Vert
w-w^{\prime}\right\Vert
\end{align*}
and so
\[
d((z,v),\mathcal{D}^{t}(w^{\prime}))\leq\left\Vert (z,v)-(z^{\prime}
,v^{\prime})\right\Vert <(1+\tau)\left\Vert w-w^{\prime}\right\Vert .
\]
Taking the supremum over $\mathcal{D}^{t}(w)\cap B\left((y,u),\dfrac{r}{3}\right)$
it yields
\[
e\left(\mathcal{D}^{t}(w)\cap B\left((y,u),\frac{r}{3}\right), \; \mathcal{D}^{t}(w^{\prime}
)\right)\leq(1+\tau)\left\Vert w-w^{\prime}\right\Vert \;\; \mbox{ for all } \;\; t\in(0,\alpha).
\]

Finally, for (iii), consider $\varepsilon>0$ and $\gamma>0$ arbitrarily such that $\gamma
<\dfrac{\varepsilon}{3}$. Since $\mathcal{S}_{p}$ is inner semicontinuous at
$(\bar{x},\bar{y},\bar{u})$ with $(\bar{y},\bar{u})\in\mathcal{S}_{p}
(\bar{x})$, there exists $r_{1}>0$ such that
\[
d\left((\bar{y},\bar{u}),\; \mathcal{S}_{p}(x)\right)<\gamma \;\; \mbox{ for all } \;\; x\in B(\bar{x},r_{1}).
\]
From the Hausdorff-lower semicontinuity  of
$\mathcal{S}_{p}^{t}(x)$ at $0^{+}$ for any $x$ in $B(\bar{x},r_{2})$
($r_{2}>0$), we have
\[
e(\mathcal{S}_{p}(x),\;\mathcal{S}_{p}^{t}(x))<\gamma.
\]
Fix now $x\in B(\bar{x},r)$ with $0<r<\min(r_{1},r_{2})$ and $(y,u)\in
V\cap\mathcal{S}_{p}^{t}(x)$ where $V:=B((\bar{y},\bar{u}),\;\gamma)$. Let
$x^{\prime}\in B(x,r^{\prime})$ with $0<r^{\prime}<r-\left\Vert x-\bar
{x}\right\Vert $ so that $x^{\prime}\in B(\bar{x},\;r)$ and thus
\[
d\left((y,u),\;\mathcal{S}_{p}^{t}(x^{\prime})\right)  \;\, \leq \;\,
\left\Vert (y,u)-(\bar{y},\bar{u})\right\Vert +d((\bar{y},\bar{u}),\;\mathcal{S}_{p}(x^{\prime  })) + e(\mathcal{S}_{p}(x^{\prime}),\;\mathcal{S}_{p}^{t}(x^{\prime}))
\; < \; 3 \; \gamma \; <\; \varepsilon
\]
and the conclusion follows. \hfill \qed
\end{proof}
\begin{corollary}\label{CorollaryF} The following assertions are satisfied:
\begin{description}
\item[(i)] Let the set-valued mapping $\mathcal{D(\cdot)}$ be H-lower semicontinuous at $\bar{x}$, while the set-valued mapping
$\mathcal{D}(\cdot,\cdot)$ is H-upper semicontinuous at $(0^+,\bar{x})$. Then, there exists a neighborhood $U$ of the point $\bar{x}$ such that for $t\downarrow0$, the function $x \mapsto \psi_{p}^{t}(x)$ is lower semicontinuous at all $x^t\in U$. 
\item[(ii)] Suppose that the set-valued mapping $\mathcal{D(\cdot)}$ is Lipschitz-like around $(\bar{x},\bar{y},\bar{u})\in \mbox{gph}\mathcal{D}$ and $\mathcal{D}(\cdot, x)$ is H-upper semicontinuous at $0^{+}$ for any $x$ around $\bar{x}$.
Furthermore, let  $\mathcal{S}_{p}(\cdot)$ be inner semicontinuous at the point $(\bar{x},\bar{y},\bar{u})\in \mbox{gph}\mathcal{S}_{p}$ and $t \rightrightarrows \mathcal{S}_{p}^{t}(x)$ be H-lower semicontinuous at $0^{+}$ for any $x$ close to $\bar{x}$.
Then, there exists a neighborhood $U$ of $\bar{x}$ such that for $t\downarrow0$, the function $x \mapsto \psi_{p}^{t}(x)$ is Lipschitz continuous around all $x^t\in U$.
\end{description}
\end{corollary}
\begin{proof}
Assertion (i) follows from a combination of Theorem \ref{prop}(i), Theorem \ref{ExistencePt}(iv), and  Theorem \ref{Lemma223}(i), while (ii) results from combining Theorem \ref{ExistencePt}(v)(a), Theorem \ref{prop}(ii), and Theorem \ref{prop}(iii).
\end{proof}


\begin{example}
Let us show that all the assumptions in Theorem \ref{prop} (i.e., in Corollary \ref{CorollaryF} as well) hold for the problem in Example \ref{ex4}. First, one can check that all hypotheses  are fulfilled at the optimal solution $\bar{x} =-1 $.  Clearly,  the set-valued mapping
$\mathcal{D}(\cdot)$  is  Lipschitz continuous around $\bar{x}$ since for any $x$, $x^{\prime }$ around $\bar{x}$, one has
\begin{equation*}\label{liplik}
e(\mathcal{D}(x),\mathcal{D}(x^{\prime }))=\left\vert x-x^{\prime
}\right\vert.
\end{equation*}
Hence, the set-valued mapping $\mathcal{D}(\cdot)$ is H-lower semicontinuous at $\bar{x}$. Next, let us show that $\mathcal{D}(\cdot,\cdot)$ is H-upper semicontinuous at $(0^+,\bar{x})$. Given $\varepsilon>0$, choose $\delta >0$  such that $U:= \left[ -1,\delta -1 \right[ \subset \left [-1,0\right[$. Taking $x\in U $ and $(y,u)\in \mathcal{D}^{t}(x) $ with $t>0$ such that $x<-t$,  we have
\[
(y,u):=(y,(u_{1},u_{1}-x))=(1,(0,-x))+(y-1,(u_{1},u_{1}))
\]
with the string of conditions
\begin{equation*}
        \left(0 \leq u_{1}<t\wedge -\dfrac{t}{u_{1}-x}\leq y-1 \leq  0\right)\vee  \left(t   \leq u_{1}\leq \bar{u}^t(x)\wedge-\dfrac{t}{u_{1}-x}\leq
y-1\leq\dfrac{t}{u_{1}}-1\right).
\end{equation*}
Since  $ \bar{u}^t(x)\leq2t$ and $u_{1}-x\geq -x > 1-\delta$, it follows that
\[
\left| u_{1}\right| \leq2t\,\,\wedge\,\, \left| y-1\right|\leq\dfrac{t}{1-\delta}.
\]
Thus for $0<t<\min\,\left((1-\delta) \varepsilon, \; \dfrac{\varepsilon}{2}\right)$,
$(y-1,\,u_{1},\,u_{1})\in\varepsilon \mathbb{B}_{3}$ so that
$$(y-1,\,u_{1},\,u_{1}-x) \in \mathcal{D}(-1)+\varepsilon \mathbb{B}_{3}.$$
 Hence for any $t$ in a neighborhood of $0^{+}$ and $x\in U$,
\[
\mathcal{D}^{t}(x)\subset\mathcal{D}(-1)+\varepsilon \mathbb{B}_{3}.
\]
This means that the set-valued mapping $\mathcal{D}(\cdot,\cdot):(t,x) \rightrightarrows \mathcal{D}^{t}(x)$ is H-upper semicontinuous at $(0^+,\bar{x})$ and so is for  the set-valued mapping $\mathcal{D}(\cdot,x): t \rightrightarrows \mathcal{D}^{t}(x)$ at $0^{+}$. Thus, the properties (i) and (ii) in Theorem \ref{prop} hold.

Now for  Theorem \ref{prop}(iii), we have
 $\mathcal{S}_{p}(0)=\left\{(1,0,0)\right\}$  and for   $x\neq 0$, $$\mathcal{S}_{p}(x)=\mathcal{D}(x)=\left\{
\begin{array}{lll}
\left\{(0,\,x,\,0)\right\}         & \text{if} & 0<x \leq 1,\\[1ex]
\left\{(1,\, 0,\, -x)\right\}  & \text{if} & -1\leq x<0.
\end{array}
\right.$$
Obviously,  		 $\mathcal{S}_p(\cdot)$ is  Lipschitz continuous around $\bar{x}$ and thus  it is inner semicontinuous at any ($\bar{x},\bar{y},\bar{u}) \in  \mbox{gph}\mathcal{S}_{p} $.
On the other hand, for $x\in\left[  -1,0\right]$,
\[
\mathcal{S}^{t}_p(x)=
\left\{  (1,\,u_{1},\,u_{1}-x)\left|\;\, 0\leq u_{1}\leq t\right.\right\},
\]
 while for $0<x\leq 1$, we have
\begin{equation*}
\mathcal{S}^{t}_p(x)=\left\{\begin{array}{lll}
\left\{\left(\dfrac{t}{x},\,x,\,0\right)\right\}               & \mbox{ if } & 0 < t\leq x ,\\
\left\{(1,\,u_{1},\,u_{1}-x)| \;\; x\leq u_{1}\leq t \right\}  & \mbox{ if } &x\leq t.
                          \end{array}\right.
\end{equation*}
So the set-valued mapping $t \rightrightarrows \mathcal{S}^{t}_{p}(x)$ is   H-lower  semicontinuous  at $0^+$ for any $x$ close to  $\bar{x}$. Indeed, since for any $x\in\left[-1,    0\right[$   and  $0<t<-x$, we have
\[
(1,\,0,\,-x)=\left(1,\,t,\,t-x\right)- \left(0,\,t,\,t\right)\in \mathcal{S}^{t}_{p}(x)+t \mathbb{B}_{3},
\]
it follows that for all $\varepsilon>0$ and  $0<t<\varepsilon$, $\mathcal{S}_{p}(x)\subset\mathcal{S}^{t}_{p}(x)+\varepsilon \mathbb{B}_{3}$.
 \hfill \qed
\end{example}

To conclude this section, let us have a general look at specific problem data requirements for \eqref{MPCC} to ensure the fulfillment of the assumptions imposed in Corollary \ref{CorollaryF}. This will give (see remark below)  specific frameworks based on problem data in \eqref{MPCC} such that the following assumptions are satisfied: 
\begin{description}
\item[(C1)]  $\mathcal{D(\cdot)}$ is H-lower semicontinuous at $\bar{x}$;
\item[(C2)]  $\mathcal{D}(\cdot,\cdot)$ is H-upper semicontinuous at $(0^+,\bar{x})$;
\item[(C3)]  $\mathcal{D(\cdot)}$ is Lipschitz-like around $(\bar{x},\bar{y},\bar{u})\in \mbox{gph}\mathcal{D}$;
\item[(C4)] $\mathcal{D}(\cdot, x)$ is H-upper semicontinuous at $0^{+}$ for any $x$ around $\bar{x}$;
\item[(C5)] $\mathcal{S}_{p}(\cdot)$ is inner semicontinuous at the point $(\bar{x},\bar{y},\bar{u})\in \mbox{gph}\mathcal{S}_{p}$;
\item[(C6)] $t \rightrightarrows \mathcal{S}_{p}^{t}(x)$ is H-lower semicontinuous at $0^{+}$ for any $x$ close to $\bar{x}$.
\end{description}

\begin{remark} We have the following scenarios for the fulfillment of assumptions (C1), \ldots, (C6):\\
{\bf (C1) and (C3)}: Observe that $(y,u)\in \mathcal{D}(x)$ if and only if $0\in \Phi(x,y,u)$ with
\[
 \Phi(x, y, u):=\phi(x, y, u) + \K \;\,\mbox{ with }\;\, \phi(x, y, u):= \left(\begin{array}{c}
                                                                                      \mathcal{L}(x, y, u)\\
                                                                                      -u\\
                                                                                      g(x, y)
                                                                                                \end{array}\right)  \;\,\mbox{ and }\;\, \K:=\left\{0_m\right\}\times \Lambda,
\]
where
$
\Lambda:=\left\{(a, b)\in \mathbb{R}^{2q}\left|\; a\geq 0,\;  b\geq 0, \; a^\top b=0\right.\right\}.
$
From \cite[Theorem 9.43 and Example 9.44]{RTW}, if
 the condition
\[
\nabla \phi(\bar{x},\bar{y},\bar{u})^{\top}v =0, \;\; v\in N_{\K}(-\phi(\bar{x},\bar{y},\bar{u}))\Longrightarrow v = 0
\]
holds at $(\bar{x},\bar{y},\bar{u})$ for any $(\bar{y},\bar{u})\in \mathcal{D}(\bar{x})$, where
   $N_{K}$ denotes the limiting normal cone to $K$, then there exist $\tau \geq 0$, $r>0$, and $\delta >0$ such that for any $x\in B(\bar{x},r)$ and any $(y,u)\in B((\bar{y},\bar{u}),\delta)$, one has
\begin{equation*}
d((y,u),\mathcal{D}(x))\leq \tau d(0,\Phi (x,y,u)).
\end{equation*}
Hence, for all  $x\in B(\bar{r},r)$, it holds that
\begin{eqnarray}
d((\bar{y},\bar{u}),\mathcal{D}(x)) &\leq &\tau d(-\phi(x,\bar{y},\bar{u}),K) \nonumber \\
&\leq &\tau \left\Vert \phi(\bar{x},\bar{y},\bar{u})-\phi(x,\bar{y},\bar{u})\right\Vert
\nonumber
\end{eqnarray}
and as $\phi$ is Lipschiz continuous around $(\bar{x},\bar{y},\bar{u})$ (since it is continuously differentiable), we get
$$d((\bar{y},\bar{u}),\mathcal{D}(x)) \leq  \tau l \left\Vert x-\bar{x} \right\Vert. $$
We conclude that $\mathcal{D}(\cdot )$ is H-lower
semicontinuous at $\bar{x}.$ Indeed, let $\varepsilon >0$,  then for all $x\in B(\bar{x}, r)$ such that $ l\tau \left\Vert x-\bar{x} \right\Vert < \varepsilon$,  one has
$
d((\bar{y},\bar{u}),\mathcal{D}(x)) \leq  \varepsilon
$
for any arbitrary  $(\bar{y},\bar{u}) \in \mathcal{D}(\bar{x})$. The conclusion follows by taking the upper bound over $\mathcal{D}(\bar{x})$.
Moreover,  $\Phi$ is obviously inner semicontinuous at $(\bar{x}, \bar{y},\bar{u}, 0)$ and  Lipschitz-like with respect to $x$ uniformly in $(y,u)$ around $(\bar{x}, \bar{y},\bar{u}, 0)$ ($\phi$ being Lipschitz continuous with respect to $x$ uniformly in $(y,u)$) then from Theorem 3.6 in \cite{Du}, the set-valued mapping $\mathcal{D}(\cdot)$ is
Lipschitz-like around $(\bar{x}, \bar{y},\bar{u})$.\\[2ex]
{\bf (C2) and (C4)}: Recall that in order to prove that a set-valued mapping $\Psi :\mathbb{R}^n \rightrightarrows \mathbb{R}^m$ defined by inequalities is  B (H)-upper semicontinuous at a point $w_0\in \mathrm{dom}\Psi$, in general we assume that there exists a compact set $C\subset \mathbb{R}^n$ such that for all parameters $w$ in a certain neighborhood
$W$ of $w_0$, the sets $\Psi(w)$ are contained in $C$. In our case, all subsets $\mathcal{D}(t,x)$ are closed so that, assuming the fulfillment of the uniform boundedness around $(0^{+},\bar{x
})$, i.e., the existence of $r$, $\alpha$, and $\gamma >0$ such that
\begin{equation}
\forall (t,x)\in \left[ 0,\alpha \right[ \times B(\bar{x},r):\mathcal{D}
(t,x)\subset \gamma \mathbb{B}_{m+q},  \label{bound}
\end{equation}
the set-valued mapping $t \rightrightarrows  \mathcal{D}(t,x)$ is H-upper
semicontinuous at $0^{+}$ for any $x$ around $\bar{x}$ and the set-valued
mapping $(t,x) \rightrightarrows  \mathcal{D}(t,x)$ is H-upper semicontinuous at $
(0^{+},\bar{x})$. Indeed, by continuity of all functions data, the set-valued
mapping $\mathcal{D}(\cdot ,x)$ (resp. $\mathcal{D}(\cdot ,\cdot )$) is
outer semicontinuous at $0^{+}$ (resp.  $(0^{+},\bar{x})$) for any $x$.
Hence by assumption (\ref{bound}),  $\mathcal{D}(\cdot ,x)$ (resp. $\mathcal{D}(\cdot ,\cdot )$) is B-upper semicontinuous at $0^{+}$ (resp. $(0^{+},\,\bar{x})$) for any $x$ (see \cite[Theorem 5.19]{RTW})
and so it is H-upper semicontinuous too from Theorem 2.1(ii).\\[2ex]
{\bf (C5)}:  Among the most important conditions ensuring the property of inner semiontinuity of $\mathcal{S}_{p}$ at $(\bar{x},\bar{y},\bar{u})$, let us first mention the Robinson condition \cite{Robinson1} for the lower-level problem \eqref{S(x)}. Note that the Robinson condition will be said to hold at $(\bar{x},\bar{y},\bar{u})$ if a strong second order sufficient condition is satisfied at this point and the linear independence constraint qualification holds at $\bar y$ for the lower-level feasible set for $x:=\bar x$. Clearly, if the Robinson condition holds at $(\bar{x},\bar{y},\bar{u})$, then $\mathcal{S}_{p}(\bar x) = \mathcal{D}(\bar x) = \{(\bar y, \bar u)\}$. As second option, $\mathcal{S}_{p}$ is inner semicontinuous at $(\bar{x},\bar{y},\bar{u})$ if this set-valued mapping is Lipschitz-like around $(\bar{x},\bar{y},\bar{u})$. Results ensuring Lipschitz-like behavior of solution
maps based on the coderivative criterion and the appropriate
second-order subdifferential of nonsmooth functions can be found in the book by Mordukhovich \cite[Chapter 4]{BS2} and the references therein.\\[2ex]
 {\bf (C6)}: In general it is difficult to check the H-lower semicontinuity of an optimal solution set-valued mapping, i.e., that for any $x$ around $\bar{x}$ and any $\varepsilon >0$ there exists $\alpha >0$, such that for all $ t\in (0,\alpha)$, one has
 \[
 \mathcal{S}_{p}(x)\subset \mathcal{S}^{t}_{p}(x)+\varepsilon \mathbb{B}_{m+q},
 \]
 without knowing the analytical description of the set-valued mapping $\mathcal{S}_{p}(x)$. However, this property is satisfied (see \cite[Theorem 1]{Zhao}
 and \cite[Theorem 2.1]{Kien}) whenever the set-valued mapping $\mathcal{D}(\cdot,x)$ is H-lower semicontinuous at $0^+$ around $\bar{x}$ (which is satisfied since for any $x$, $\mathcal{D}(0,x)=\mathcal{D}(x)\subset \mathcal{D}(t,x) \,\mbox{ for all }\, t>0$) and the set-valued mapping $t \rightrightarrows \mathcal{S}^{t}_{p}(x)$ is uniformly closed and bounded (nonempty) near $t=0^+$ around $\bar{x}$ as well as the following condition holds for any $x$ around $\bar{x}$: for any $\varepsilon>0$, there exist $\alpha, \delta >0$ such that for all $t\in (0,\alpha)$ and any $(y,u)\in \mathcal{D}^{t}(x)$ satisfying $d((y,u),\mathcal{S}^{t}_{p}(x))>\varepsilon$, $F(x,y)\leq \psi^{t}_{p}(x) -\delta$.
\end{remark}
\section{Computing C-stationary points}\label{Computing C-stationary points}
Our basic assumption in this section is that at Step 1 of Algorithm \ref{algorithm 1}, we are computing stationary points of \eqref{RMPCCt} for $t:=t_k$, with the aim to construct a framework ensuring that the resulting sequence converges to a C-stationary point of \eqref{MPCC}. To proceed, we first introduce the C-stationarity concept of the latter problem. 
\begin{definition}\label{C-concept}
A point $\bar{x}$ will be said to be C-stationary point for \eqref{PBP} if there exists a vector $\left(\bar{y},\bar{u}, \alpha,\beta,\gamma\right)$ such that we have the relationships
\begin{eqnarray}
(\bar{x}, \bar{y},\bar{u})\in \mbox{gph}\,\mathcal{S}_{p},\label{St0}\\
\nabla_{x}F(\bar{x},\bar{y})+\sum\limits_{i=1}^{p}\alpha_{i}\nabla G_{i}
(\bar{x})+\sum\limits_{l=1}^{m}\beta_{l}\nabla_{x}\mathcal{L}_{l}(\bar{x}
,\bar{y},\bar{u})+\sum\limits_{i=1}^{q}\gamma_{i}\nabla_{x}g_{i}(\bar{x}
,\bar{y})=0,\label{St3}\\
\nabla_{y}F(\bar{x},\bar{y})+\sum\limits_{l=1}^{m}\beta_{l}\nabla
_{y}\mathcal{L}_{l}(\bar{x},\bar{y},\bar{u})+\sum\limits_{i=1}^{q}\gamma
_{i}\nabla_{y}g_{i}(\bar{x},\bar{y})=0,\label{St4}\\[1ex]
\forall j=1, \ldots, p:\;\; \alpha_j\geq0,\;\; G_j(\bar{x})\leq 0, \;\; \alpha_jG_j(\bar{x})=0,\label{St5}\\[1.5ex]
\forall i\in \nu:\;\, \sum\limits_{l=1}^{m}\beta_{l}\nabla_{y_{l}}g_{i}(\bar{x},\bar{y})=0, \;\;  \forall i\in \eta:\;\,\gamma_i=0,\label{St6}\\
\forall i\in\theta: \;\;\, \gamma_{i}\sum\limits_{l=1}^{m}\beta_{l}\nabla_{y_{l}}g_{i}(\bar{x},\bar{y})\geq0,\label{St8}
\end{eqnarray}
where
the index sets $\eta$, $\theta$, and $\nu$ are respectively defined as follows:
\[
\begin{array}{rlrll}
\eta &  := &\eta(\bar{x},\bar{y},\bar{u}) &:= &\left\{  i\in\left\{1,...,q\right\}\left|\;\; \bar{u}_{i}=0,\text{ }g_{i}(\bar{x},\bar{y})<0\right.\right\},\\
\theta &  := &\theta(\bar{x},\bar{y},\bar{u}) &:= &\left\{i\in\left\{1,...,q\right\}\left|\;\;\bar{u}_{i}=0,\text{ }g_{i}(\bar{x},\bar{y})=0\right.\right\},\\
\nu &  := &\nu(\bar{x},\bar{y},\bar{u}) &:=& \left\{  i\in\left\{  1,...,q\right\}\left|\;\; \bar{u}_{i}>0,\text{ }g_{i}(\bar{x},\bar{y})=0\right.\right\}.
\end{array}
\]
\end{definition}
Similarly, we can define the M- and S-stationarity by replacing condition \eqref{St8} by
\[
\begin{array}{c}
\forall i\in\theta: \;\;\, (\gamma_{i}<0\wedge\sum\limits_{l=1}^{m}\beta
_{l}\nabla_{y_{l}}g_{i}(\bar{x},\bar{y})<0)\vee\gamma_{i}\sum\limits_{l=1}%
^{m}\beta_{l}\nabla_{y_{l}}g_{i}(\bar{x},\bar{y})=0\\
\left(\mbox{resp. }\;\, \forall i\in\theta:\;\;\, \gamma_{i}\leq0\wedge\sum\limits_{l=1}^{m}%
\beta_{l}\nabla_{y_{l}}g_{i}(\bar{x},\bar{y})\leq0\right).
\end{array}
\]
Obviously, the following string of implications is obtained directly from these definitions:
\[
S\text{-stationarity }\Longrightarrow\text{ }M\text{-stationarity
}\Longrightarrow\text{ }C\text{-stationarity.}
\]
These concepts were introduced in \cite{D2} in the context of \eqref{PBP}. Further stationarity concepts for pessimistic bilevel optimization and relevant details can be found in the latter reference. As for optimistic bilevel programs and the closely related MPCCs, see \cite{DempeZemkohoKKTRefNonsmooth,DempeZemkohoKKTRef,HoheiselEtAlComparison2013} and references therein.

Next, we recall a result, which provides a framework ensuring that a local optimal solution of problem \eqref{PBP} satisfies the C-stationarity conditions. To proceed, we need to introduce some assumptions. First, the upper-level regularity condition will be said to hold at $\bar{x}$ if
\begin{equation}\label{UMFCQ}
\left[\alpha\geq 0, \;\;\; \alpha^\top G(\bar{x})=0,\;\;\; \nabla G(\bar{x})^\top \alpha =0\right] \;\;\;\Longrightarrow  \;\;\; \alpha =0.
\end{equation}
Obviously, this regularity concept corresponds to the Mangasarian–Fromovitz constraint qualification for the upper-level constraint. For the remaining qualification conditions, we introduce the C-qualification conditions {{at the point $(\bar{x},\bar{y}, \bar{u})$, which are defined by}}
\begin{align*}
(\beta,\gamma)  &  \in\Lambda^{ec}(\bar{x},\bar{y},\bar{u},0)\implies
\beta=0,\;\; \gamma=0,\medskip \tag{$A_{1}^{c}$}\\
(\beta,\gamma)  &  \in\Lambda_{y}^{ec}(\bar{x},\bar{y},\bar{u},0)\implies
\nabla_{x}\mathcal{L}(\bar{x},\bar{y},\bar{u})^{T}\beta+\nabla_{x}g(\bar
{x},\bar{y})^{T}\gamma=0,\medskip \tag{$A_{2}^{c}$}\\
(\beta,\gamma)  &  \in\Lambda_{y}^{ec}(\bar{x},\bar{y},\bar{u},0)\implies
\beta=0,\;\;\gamma=0, \tag{$A_{3}^{c}$}\text{ }
\end{align*}
with the C-multiplier set $\Lambda^{ec}(\bar{x},\bar{y},\bar{u},0)$ resulting from setting $v=0$ in
\begin{equation}\label{C-multipliers set}
\Lambda^{ec}(\bar{x},\bar{y},\bar{u},v) := \left\{(\beta,\gamma)\in\mathbb{R}^{m+q}\left|\;\;\begin{array}{l}
                                                                                          \nabla_{y}g_{\nu}(\bar{x},\bar{y})\beta=0,\;\; \gamma_{\eta}=0\\
                                                                                          \forall i\in\theta: \;\; \gamma_{i}\left(\nabla_{y}g_{i}(\bar{x},\bar{y})\right)\beta\geq 0\\
                                                                                          v+\nabla_{x,y}\mathcal{L}(\bar{x},\bar{y},\bar{u})^{\top}
                                                                                          \beta + \nabla g(\bar{x},\bar{y})^{\top}\gamma=0
                                                                                         \end{array}
   \right.\right\},
\end{equation}
while $\Lambda_{y}^{ec}(\bar{x},\bar{y},\bar{u},0)$ can be obtained by replacing the derivatives w.r.t. $(x, y)$ in the last equation in the set \eqref{C-multipliers set} by the derivatives of the same functions w.r.t. $y$ only. Obviously,
\[
(A_{1}^{c})\impliedby(A_{3}^{c})\implies(A_{2}^{c}).
\]
Similarly to the $M$- and $S$-stationarity concepts introduced at the beginning of this section, we can define $M$- and $S$-qualification conditions.
{{
\begin{theorem}[\cite{D2}]Let $\bar{x}$ be a local optimal solution for problem
\eqref{PBP}. Suppose that $\bar{x}$ is upper-level regular and  $\mathcal{S}_p$ is inner
semicontinuous at $(\bar{x},\bar{y},\bar{u})$, where the conditions $(A_{1}^{c})$ and
$(A_{2}^{c})$ are also assumed to hold. Then there exists $\left(\alpha,\beta,\gamma\right)$  such that the optimality conditions \eqref{St0}--\eqref{St8} are satisfied.
\end{theorem}
Note that both the convexity of the lower-level problem and the fulfillment of the regularity condition \eqref{SCQ} are needed here to ensure that the result holds. However, it is assumed that both conditions are satisfied throughout the paper. Hence, they are not explicitly stated in the result.}}

Next, we introduce the  necessary optimality conditions for the regularized problem \eqref{RMPCCt}.
For a fixed number $t>0$, a point $x^t$ will be said to \emph{satisfy the necessary optimality conditions} (or \emph{be a stationary point}) for problem \eqref{RMPCCt} if there exists a vector $(y^{t},u^{t}, \alpha^{t}, \beta^{t}, \gamma^{t}, \mu^{t},\delta^{t})$ such that 
\begin{eqnarray}
(x^{t},y^{t},u^{t})\in\mbox{gph }\mathcal{S}^{t}_p, \label{Er0}\\[1ex]
\nabla_{x}F(x^{t},y^{t})+\nabla G(x^{t})^{\top}\alpha^{t}+\nabla
_{x}\mathcal{L}(x^{t},y^{t},u^{t})^{\top}\beta^{t}-\sum\limits_{i=1}
^{q}\left(  \gamma_{i}^{t}-\delta_{i}^{t}u_{i}^{t}\right)  \nabla_{x}
g_{i}(x^{t},y^{t})=0,\label{Er1}\\
\nabla_{y}F(x^{t},y^{t})+\nabla_{y}\mathcal{L}(x^{t},y^{t},u^{t})^{\top}
\beta^{t}-\sum\limits_{i=1}^{q}\left(\gamma_{i}^{t}-\delta_{i}^{t}u_{i}
^{t}\right)  \nabla_{y}g_{i}(x^{t},y^{t})=0,\label{Er1-f1}\\
\forall i=1, \ldots, q:\;\, \nabla_{y}g_{i}(x^{t},y^{t})\beta^{t}+\mu_{i}^{t}+\delta_{i}^{t}g_{i}
(x^{t},y^{t})=0,\label{Er1-f2}\\[1ex]
\forall j=1, \ldots, p: \;\; \alpha_{j}^{t} \geq 0,\;\; G_j(x^t)\leq 0, \;\; \alpha_{j}^{t}G_{j}(x^{t})=0,\label{Er011}\\[1ex]
\forall i=1, \ldots, q:\;\; \gamma^t_i\geq 0, \;\; g_i(x^t, y^t)\leq 0, \;\; \gamma^t_i g_i(x^t, y^t)=0,\label{Er1-f3}\\[1ex]
\forall i=1, \ldots, q:\;\; \mu^t_i\geq 0, \;\; u^t_i\geq 0, \;\; \mu^t_i u^t_i=0,\label{Er1-f4}\\[1ex]
\forall i=1, \ldots, q:\;\; \delta^t_i\geq 0, \;\; -u^t_ig_i(x^t, y^t)\leq t, \;\; \delta^t_i \left(u^t_ig_i(x^t, y^t)+t\right)=0. \label{Er5}
\end{eqnarray}

\begin{theorem}[\cite{Zemkoho}]\label{Optimality conditions for RMPCC(t)}For a given $t>0,$
let $x^{t}$ be an upper-level regular local optimal solution for problem \eqref{RMPCCt}.
Assume that the set-valued mapping $\mathcal{S}^{t}_p$ is inner
semicontinous at $(x^{t},y^{t},u^{t})\in\mathrm{gph}\,\mathcal{S}^{t}_p$ and
the following qualification condition holds at the point  $(x^{t},y^{t},u^{t})$:
\begin{equation}
\left.
\begin{array}
[c]{r}
\nabla_{y}\mathcal{L}(x^{t},y^{t},u^{t})^{\top}\beta^{t}-\sum\limits_{i=1}
^{q}\nabla_{y}g_{i}(x^{t},y^{t})\left(  \gamma_{i}^{t}-\delta_{i}^{t}u_{i}
^{t}\right)  =0\\
\nabla_{y}g_{i}(x^{t},y^{t})\beta^{t}+\mu_{i}^{t}+\delta_{i}^{t}g_{i}
(x^{t},y^{t})=0\\
\gamma^{t}\geq0,\text{ }\delta^{t}\geq0,\text{ }\mu^{t}\geq0\\
\mu_{i}^{t}u_{i}^{t}=0,\text{ \ }\gamma_{i}^{t}g_{i}(x^{t},y^{t}
)=0,\text{\ }\delta_{i}^{t}\left(  u_{i}^{t}g_{i}(x^{t},y^{t})+t\right)  =0
\end{array}
\right\}  \Longrightarrow\left\{
\begin{array}
[c]{l}
\beta^{t}=0_{m},\\
\\
\delta^{t}=\gamma^{t}=\mu^{t}=0_{q}.
\end{array}
\right. \label{CQ1}
\end{equation}
Then $x^t$ is a stationary point for problem \eqref{RMPCCt}; i.e., there exists a vector $(y^{t},u^{t}, \alpha^{t}, \beta^{t}, \gamma^{t}, \mu^{t},\delta^{t})$ such that the optimality conditions  \eqref{Er0}--\eqref{Er5} are satisfied.
\end{theorem}

%
To establish the convergence result of Algorithm \ref{algorithm 1} in this context, the following index sets, defined for $(x,(y,u))\in X\times\mathcal{D}^{t}(x)$ with $t>0$, are needed:
\[
\begin{array}{rll}
I_{G}(x)&:=&\left\{  i\in\{1,...,p\}| \;\; G_{i}(x)=0\right\},\\
 I_{u}(x,y,u)&:=&\left\{  i\in\{1,...,q\}| \;\;u_{i}=0\right\},\\
I_{g}(x,y,u)&:=&\left\{  i\in\{1,...,q\}| \;\;g_{i}(x,y)=0\right\},\\
I_{ug}(x,y,u)&:=&\left\{  i\in\{1,...,q\}| \;\; u_{i}g_{i}(x,y)+t=0\right\}.
\end{array}
\]
Clearly, $I_{g}(x,y,u)\cap I_{ug}(x,y,u)=\varnothing$ and $I_{u}(x,y,u)\cap I_{ug}(x,y,u)= \varnothing$. Also, for the convergence result, we will use the $m$--counterparts, $(A_{1}^{m})$ and $(A_{2}^{m})$,  of assumptions $(A_{1}^{c})$ and $(A_{2}^{c})$, which can be obtained by simply replacing the C-multipliers set \eqref{C-multipliers set} by the corresponding M-multipliers set. One can easily check that $(A_{1}^{m})$ and $(A_{2}^{m})$ are weaker than $(A_{1}^{c})$ and $(A_{2}^{c})$, respectively.


\begin{theorem}\label{ConvergenceResult}
Let $t_{k} \downarrow 0$ and $(x^{k})_k$ be a sequence such that $x^k$ is a stationary point of problem \eqref{RMPCCt} for $t:=t^k$ with
$x^{k}\rightarrow\bar{x}$ as $k\rightarrow\infty$ such that $\mathcal{S}_{p}(\bar{x})$ is
nonempty and compact. Assume that $\bar{x}$ is upper-level regular, and $(A_{1}^{m})$ and $(A_{2}^{m})$ hold at all $(\bar{x},y,u)$ with $(y,u)\in\mathcal{S}_{p}(\bar{x})$. Furthermore, let
\begin{equation}
\underset{k\rightarrow\infty}{\lim} e\left(\mathcal{S}^{t_{k}}_p(x^{k}), \;\mathcal{S}_{p}(\bar{x})\right) = 0.\label{e}
\end{equation}
Then $\bar{x}$ is a C-stationary point.
\end{theorem}
\begin{proof}
Since $x^{k}$ is a stationary point of \eqref{RMPCCt} for $t:=t_k$, there exist a vector
$(y^{k},u^{k})\in\mathcal{S}^{t_{k}}_p(x^{k})$ and multipliers $(\alpha
^{k},\beta^{k},\mu^{k},\gamma^{k},\delta^{k})$ such that the relationships
\begin{equation}\label{S1}
\left\{
\begin{array}{l}
\nabla_{x}F(x^{k},y^{k})+\nabla G^{T}(x^{k})\alpha^{k}+\nabla_{x}
\mathcal{L}(x^{k},y^{k},u^{k})^{T}\beta^{k}-\sum\limits_{i=1}^{q}(\gamma
_{i}^{k}-\delta_{i}^{k}u_{i}^{k})\nabla_{x}g_{i}(x^{k},y^{k})=0,\\[1ex]
\multicolumn{1}{l}{\nabla_{y}F(x^{k},y^{k})+\nabla_{y}\mathcal{L}(x^{k}
,y^{k},u^{k})^{T}\beta^{k}-\sum\limits_{i=1}^{q}(\gamma_{i}^{k}-\delta_{i}
^{k}u_{i}^{k})\nabla_{y}g_{i}(x^{k},y^{k})=0,}\\[1ex]
\multicolumn{1}{l}{\nabla_{y}g_{i}(x^{k},y^{k})\beta^{k}+\mu_{i}^{k}%
+\delta_{i}^{k}g_{i}(x^{k},y^{k})=0,\;\; i=1,...,q}
\end{array}
\right.
\end{equation}
hold together with the following complementarity conditions:
\begin{equation*}\label{S2}
\left\{\begin{array}{l}
\alpha_{i}^{k}\geq0,\;\, G_{i}(x^{k})\leq0,\;\, \alpha_{i}^{k}G_{i}(x^{k})=0,\;, i=1,...,p,\\[1ex]
\mu_{i}^{k}\geq0,\;\, \gamma_{i}^{k}\geq0,\;\, \delta_{i}^{k}\geq0, \;\, i=1,...,q,\\[1ex]
\mu_{i}^{k}u_{i}^{k}=0,\;\; \gamma_{i}^{k}g_{i}(x^{k},y^{k})=0,\;\; \delta_{i}^{k}(u_{i}^{k}g_{i}(x^{k},y^{k})+t_{k})=0, \;\, i=1,...,q.
\end{array}
\right.
\end{equation*}
Hence, it follows that
\begin{equation*}
\nabla_{y}g_{i}(x^{k},y^{k})\beta^{k}=\left\{
\begin{array}{lll}
-\mu_{i}^{k} & \mbox{ if } & i\in\mathrm{supp}\mu^{k},\\[1ex]
-\delta_{i}^{k} g_{i}(x^{k},y^{k})& \mbox{ if } & i\in\mathrm{supp}\delta^{k},\\[1ex]
0 & \mbox{ if } & i\notin (\mathrm{supp}\mu^{k}\cup\mathrm{supp}\delta^{k}),
\end{array}
\right. \label{S2b}
\end{equation*}
and
\begin{align}
\alpha_{i}^{k} &  \geq0,\text{ \ }\mathrm{supp}\alpha^{k}\subset I_{G}
(x^{k}),\medskip\label{S3}\\
\mu_{i}^{k} &  \geq0,\text{ \ }\mathrm{supp}\mu^{k}\subset I_{u}(x^{k}
,y^{k},u^{k}),\medskip\label{S4}\\
\gamma_{i}^{k} &  \geq0,\text{ \ }\mathrm{supp}\gamma^{k}\subset I_{g}
(x^{k},y^{k},u^{k}),\medskip\label{S5}\\
\delta_{i}^{k} &  \geq0,\text{ \ }\mathrm{supp}\delta^{k}\subset I_{ug}
(x^{k},y^{k},u^{k}),\label{S6}
\end{align}
where ${{\mathrm{supp}z := \left\lbrace i |\,\, z_i \neq 0\right\rbrace}}$ is the support of a vector $z \in \mathbb{R}^n$.

Next, observe that the continuity of the constraint functions $G$ and $g$ leads to
\begin{equation}
I_{G}(x^{k})\subseteq I_{G},\text{ }I_{u}(x^{k},y^{k},u^{k})\subseteq
\theta\cup\eta,\ I_{g}(x^{k},y^{k},u^{k})\subseteq\theta\cup\nu
,\label{S7}%
\end{equation}
for $k$ sufficiently large; here
$
I_{G}:=\left\{  i\in\left\{  1,...,p\right\}|  \; \; G_{i}(\bar{x})=0\right\}.
$
Now, we can introduce some new Lagrange multipliers by setting
\begin{equation}\label{S8}
\tilde{\gamma}_{i}^{k}:=\left\{
\begin{array}{lll}
\gamma_{i}^{k} & \mbox{ if }& i\in\mathrm{supp}\gamma^{k},\\
-\delta_{i}^{k}u_{i}^{k}& \mbox{ if }& i\in\mathrm{supp}\delta^{k}\backslash\eta,\\
0& & {{\mbox{ otherwise,  }}}
\end{array}
\right. \;\, \mbox{ and } \;\,
\tilde{\mu}_{i}^{k}:=\left\{
\begin{array}{lll}
\mu_{i}^{k}& \mbox{ if }& i\in\mathrm{supp}\mu^{k},\\
\delta_{i}^{k}g_{i}(x^{k},y^{k})& \mbox{ if }& i\in\mathrm{supp}\delta^{k}\backslash\nu,\\
0& & {{\mbox{ otherwise.  }}}
\end{array}
\right.
\end{equation}
Hence, from the optimality conditions in (\ref{S1}), we have the system of equations
\begin{equation}
\left\{
\begin{array}{l}
\nabla_{x}F(x^{k},y^{k})+\nabla G(x^{k})^{T}\alpha^{k}+\nabla_{x}\mathcal{L}(x^{k},y^{k},u^{k})^{T}\beta^{k}\\[0.5ex]
 \qquad \qquad \qquad \qquad \qquad \qquad \qquad \;\;\; -\, \sum\limits_{i=1}^{q}\tilde{\gamma}_{i}
^{k}\nabla_{x}g_{i}(x^{k},y^{k}) + \sum\limits_{i\in\eta
}\delta_{i}^{k}u_{i}^{k}\nabla_{x}g_{i}(x^{k},y^{k})=0,\\[1ex]
\multicolumn{1}{l}{\nabla_{y}F(x^{k},y^{k})+\nabla_{y}\mathcal{L}(x^{k}
,y^{k},u^{k})^{T}\beta^{k}-\sum\limits_{i=1}^{q}\tilde{\gamma}_{i}
^{k}\nabla_{y}g_{i}(x^{k},y^{k})+\sum\limits_{i\in\eta}\delta_{i}^{k}u_{i}
^{k}\nabla_{y}g_{i}(x^{k},y^{k})=0,}
\end{array}
\right. \label{S10}
\end{equation}
which is accompanied by the conditions
\begin{equation}
{{\nabla_{y}g_{i}(x^{k},y^{k})\beta^{k}=\left\{
\begin{array}{lll}
-\tilde{\mu}_{i}^{k} & \mbox{ if } & i\in\mathrm{supp}\mu^{k}
\cup\mathrm{\ supp}\delta^{k}\backslash\nu,\\
-\delta_{i}^{k}g_{i}(x^{k},y^{k}) & \mbox{ if } & i\in\nu,\\
0 & \mbox{ if } & i\notin(\mathrm{supp}\mu^{k}\cup\mathrm{supp}\delta^{k}).
\end{array}
\right. }}\label{S11}
\end{equation}
On the other hand, from (\ref{e}), there exists $(z^{k},w^{k})\in
\mathcal{S}_{p}(\bar{x})$ such that
\[
\left\Vert y^{k}-z^{k}\right\Vert \underset{k \rightarrow
	\infty} \rightarrow0\text{ and }\left\Vert
u^{k}-w^{k}\right\Vert \underset{k \rightarrow
	\infty} \rightarrow 0,
\]
and since $\mathcal{S}_{p}(\bar{x})$ is compact, the sequence $(y^{k}
,u^{k})_{k}$ (up to a subsequence) converges to some  point $(\bar{y},\bar
{u})\in\mathcal{S}_{p}(\bar{x})$. Next, we are going to show that the sequence
$(\chi_{k}):=
\begin{pmatrix}
\alpha^{k},\beta^{k},\tilde{\mu}^{k},\tilde{\gamma}^{k},\delta
_{\eta\cup\nu}^{k}
\end{pmatrix}
_{k}$ is bounded. Assume that this is not the case and consider the sequence
$\left(\dfrac{\chi_{k}}{\left\Vert \chi_{k}\right\Vert }\right)_{k}$ converging to a
nonvanishing vector $\chi:=(\alpha,\beta,\tilde{\mu},\tilde{\gamma
},\delta_{\eta\cup\nu})$ (otherwise, a suitable subsequence is chosen).
Dividing (\ref{S10}) by $\left\Vert \chi_{k}\right\Vert $ and taking the limit
$k\rightarrow\infty$ while taking into account the continuous differentiability of all
involved functions and the fact that $\bar{u}_{i}=0$ for $i\in\eta$,
\begin{equation}\label{CQ}
\nabla G(\bar{x})^{T}\alpha+\nabla_{x}\mathcal{L}(\bar{x},\bar{y},\bar{u})^{T}\beta-\nabla_{x}g^{T}(\bar{x},\bar{y})\tilde{\gamma}=0
\mbox{ and }
\nabla_{y}\mathcal{L}(\bar{x},\bar{y},\bar{u})^{T}\beta-\nabla_{y}g^{T}(\bar{x},\bar{y})\tilde{\gamma}=0.
\end{equation}
Let us show now that we have the following inclusions for any $k$ sufficiently large:
\begin{equation}
\mathrm{supp}\alpha\subset I_{G}(x^{k})\subset I_{G},\label{S12}%
\end{equation}%
\begin{equation}
\mathrm{supp}\tilde{\mu}\subset I_{u}(x^{k},y^{k},u^{k})\cup I_{ug}%
(x^{k},y^{k},u^{k})\backslash\nu\subset\theta\cup\eta,\label{S13}%
\end{equation}%
\begin{equation}
\mathrm{supp}\tilde{\gamma}\subset I_{g}(x^{k},y^{k},u^{k})\cup I_{ug}%
(x^{k},y^{k},u^{k})\backslash\eta\subset\theta\cup\nu.\label{S14}%
\end{equation}
Indeed, for (\ref{S12}), let $i\in\mathrm{supp}\alpha$ {{so that}} $\alpha_{i}>0$.
Then for any $k$ large enough, $\alpha_{i}^{k}>0$; i.e., $i\in
\mathrm{supp}\alpha^{k}\subset I_{G}(x^{k})$ by (\ref{S3}). Given now $i\in
I_{G}(x^{k})$ for any $k$ sufficiently large so that, there exists $k_{0}%
\in\mathbb{N}$ such that for any $k\geq k_{0},$ $G_{i}(x^{k})=0.$ Passing to
the limit as $k\rightarrow\infty$, we get $G_{i}(\bar{x})=0$ ($G_{i}$ being
continuous) and so $i\in I_{G}(\bar{x}).$ For (\ref{S13}): take $i\in
\mathrm{supp}\tilde{\mu}$ thus for such $k$, $\tilde{\mu
}_{i}^{k}>0$ i.e., $i\in\mathrm{supp}\tilde{\mu}^{k}.$ Then for all $k$
sufficiently large, $i\in\mathrm{supp}\mu^{k}\cup\mathrm{supp}\delta
^{k}\backslash\nu.$ Hence by properties (\ref{S4}) and (\ref{S6}), $i\in
I_{u}(x^{k},y^{k},u^{k})\cup I_{ug}(x^{k},y^{k},u^{k})\backslash\nu.$ And
from (\ref{S7}), we have $I_{u}(x^{k},y^{k},u^{k})\subset\theta\cup\eta$. Now if
$u_{i}^{k}g_{i}(x^{k},y^{k})+t_{k}=0$ for $k$ large enough and
$i\notin\nu$, taking the limit leads to $\bar{u}_{i}g_{i}(\bar
{x},\bar{y})=0$ and thus $i\in\theta\cup\eta.$ Similarly, we get (\ref{S14}). On the
other hand,
\begin{align*}
\alpha_{j} \geq0,\text{ \ }G_{j}(\bar{x})\leq0, & \;\; \alpha_{j}G_{j}(\bar{x}
)=0 \;\, \mbox{ for }\;\, j=1,...,p,\medskip \\
\delta_{\eta\cup\nu} &  =\lim\delta_{\eta\cup\nu}^{k}\geq0,\medskip\\
\tilde{\gamma}_{i} &  =\lim\tilde{\gamma}_{i}^{k}=0\text{ \ \ }\forall i\in\eta,\medskip\\
\nabla_{y}g_{\nu}(\bar{x},\bar{y})\beta &  =\lim\nabla_{y}g_{\nu}(x^{k},y^{k})\frac{\beta^{k}}{\left\Vert \chi
_{k}\right\Vert } =-\lim{{\frac{\delta_{\nu}^{k}}{\left\Vert \chi
_{k}\right\Vert }g_{\nu}(x^{k},y^{k}
)=-\delta_{\nu}g_{\nu}(\bar{x},\bar{y})=0.}}
\end{align*}
Let us now prove that for all $i\in\theta$, it holds that
\[
(\tilde{\gamma}_{i}>0\wedge\nabla_{y}g_{i}(\bar{x},\bar{y})\beta
<0)\vee\tilde{\gamma}_{i}\nabla_{y}g_{i}(\bar{x},\bar{y})\beta=0.
\]
 Assume that $\tilde{\gamma}_{i}<0$ or $\nabla_{y}
g_{i}(\bar{x},\bar{y})\beta>0$ for some $i\in\theta$ with
$\tilde{\gamma}_{i}\nabla_{y}g_{i}(\bar{x},\bar{y})\beta\neq0.$ If
$\tilde{\gamma}_{i}<0$ then for any $k$ large enough, $\tilde{\gamma}_{i}
^{k}<0$ and so from (\ref{S8}), $i\in\mathrm{supp}\delta^{k}$. Moreover, as $\mathrm{supp}\delta^{k}\cap\, \mathrm{supp}\mu^{k} =\varnothing$, $i\notin\mathrm{supp}\mu^{k}$  and it follows  from equation (\ref{S11}) that we have
\[
\nabla_{y}g_{i}\left(x^{k},y^{k}\right)\frac{\beta^{k}}{\left\Vert \chi
_{k}\right\Vert }=-\frac{\delta_{i}^{k}}{\left\Vert \chi_{k}\right\Vert }
g_{i}\left(x^{k},y^{k}\right).
\]
Hence, at the limit, we get $\nabla_{y}g_{i}(\bar{x},\bar{y})\beta=0$ since $g_{i}(x^{k},y^{k})$ converges to $g_{i}(\bar{x},\bar{y})=0$
($i\in\theta$) and the sequence $\left(\dfrac{\delta_{i}^{k}}{\left\Vert \chi
_{k}\right\Vert}\right)$ is bounded. This leads to a contradiction since
$\tilde{\gamma}_{i}\nabla_{y}g_{i}(\bar{x},\bar{y})\beta\neq0$ by
assumption. Similarly, if $\nabla_{y}g_{i}(\bar{x},\bar{y})\beta>0,$
for all $k$ large enough, $\nabla_{y}g_{i}(x^{k},y^{k})\beta^{k}>0$
so from (\ref{S11}), $i\in\mathrm{supp}\delta^{k}.$ Then $\dfrac
{\tilde{\gamma}_{i}^{k}}{\left\Vert \chi_{k}\right\Vert }=-\dfrac{\delta
_{i}^{k}}{\left\Vert \chi_{k}\right\Vert }u_{i}^{k}$ converges to
$\tilde{\gamma}_{i}=0$ ($\bar{u}_{i}=0$), which leads to a contradiction too.
Consequently, we have
\[
\left\{
\begin{array}{l}
\nabla_{y}\mathcal{L}(\bar{x},\bar{y},\bar{u})^{T}\beta-\nabla
_{y}g^{T}(\bar{x},\bar{y})\tilde{\gamma}=0,\\
\nabla_{y}g_{\nu}(\bar{x},\bar{y})\beta=0,\text{ \ }\tilde{\gamma}_{\eta}=0,\\
(\tilde{\gamma}_{i}>0\wedge\nabla_{y}g_{i}(\bar{x},\bar{y})\beta<0)\vee\tilde{\gamma}_{i}\nabla_{y}g_{i}(\bar{x},\bar{y})\beta
=0\; \mbox{ for }\; i\in\theta,
\end{array}
\right.
\]
so that $(-\beta,\tilde{\gamma})\in\Lambda_{y}^{em}(\bar{x},\bar
{y},\bar{u},0)$ and by condition  $(A_{2}^{m})$, we get
\[
\nabla_{x}\mathcal{L}(\bar{x},\bar{y},\bar{u})^{T}\beta-\nabla
_{x}g^{T}(\bar{x},\bar{y})\tilde{\gamma}=0
\]
using (\ref{CQ}), which yields $\nabla G(\bar{x})^{T}\alpha=0$; this implies that $\alpha=0$ ($\bar{x}$ being upper-level regular). Thus,
\[
\left\{
\begin{array}{l}
\nabla_{x,y}\mathcal{L}(\bar{x},\bar{y},\bar{u})^{T}\beta-\nabla g^{T}(\bar{x},\bar{y})\tilde{\gamma}=0,\\
\nabla_{y}g_{\nu}(\bar{x},\bar{y})\beta=0,\text{ \ }\tilde{\gamma}_{\eta}=0,\\
(\tilde{\gamma}_{i}>0\wedge\nabla_{y}g_{i}(\bar{x},\bar{y})\beta<0)\vee\tilde{\gamma}_{i}\nabla_{y}g_{i}(\bar{x},\bar{y})\beta
=0\; \mbox{ for }\; i\in\theta;
\end{array}
\right.
\]
that is, $(-\beta,\tilde{\gamma})\in\Lambda^{em}(\bar{x},\bar{y},
\bar{u},0)$ and by the condition $(A_{1}^{m}),$ we get $(\beta,
\tilde{\gamma})=0.$ Hence $(\alpha,\beta,\tilde{\mu},\tilde{\gamma}
)=0$ (observe that $\tilde{\mu}=0$ from (\ref{S11})) and we conclude that
$\delta_{\eta\cup\nu}\neq0.$ Assume without loss of generality that there
exists $i\in\eta$ such that $\delta_{i}>0.$ Thus for $k$ large enough,
$i\in\mathrm{supp}\delta_{i}^{k}\backslash\nu$ and from (\ref{S8}),
$\tilde{\mu}_{i}^{k}:=\delta_{i}^{k}g_{i}(x^{k},y^{k})$ so $\dfrac{\tilde{\mu
}_{i}^{k}}{\left\Vert \chi_{k}\right\Vert }$ converges to $\tilde{\mu}
_{i}=\delta_{i}g_{i}(\bar{x},\bar{y})<0$ (since $i\in\eta$) and the
contradiction follows (with the fact that $\tilde{\mu}_{i}=0$). Similary, having
$\delta_{\nu}\neq0$ implies that we get a contradiction.

Consequently, the sequence $(\alpha^{k},\beta^{k},\tilde{\mu}
^{k},\tilde{\gamma}^{k},\delta_{\eta\cup\nu}^{k})_{k}$ is bounded.
Consider $(\bar{\alpha},\bar{\beta},\overline{\tilde{\mu}},\overline
{\tilde{\gamma}},\bar{\delta}_{\eta\cup\nu})$ its limit (up to a
subsequence) so that we have clearly
\[
\mathrm{supp}\bar{\alpha}\subset I_{G},\text{ }\mathrm{supp}\overline
{\tilde{\mu}}\subset\theta\cup\eta,\ \mathrm{supp}\overline{\tilde{\gamma}
}\subset\theta\cup\nu.
\]
Taking the limit in (\ref{S10}), we obtain
\[
\left\{
\begin{array}{l}
\nabla_{x}F(\bar{x},\bar{y})+\nabla G(\bar{x})^{T}\bar{\alpha}+\nabla
_{x}\mathcal{L}(\bar{x},\bar{y},\bar{u})^{T}\bar{\beta}-\sum
\limits_{i=1}^{q}\overline{\tilde{\gamma}}_{i}\nabla_{x}g_{i}(\bar{x},\bar
{y})=0,\\[1ex]
\multicolumn{1}{l}{\nabla_{y}F(\bar{x},\bar{y})+\nabla_{y}\mathcal{L}(\bar
	{x},\bar{y},\bar{u})^{T}\bar{\beta}-\sum\limits_{i=1}^{q}\overline
	{\tilde{\gamma}}_{i}\nabla_{y}g_{i}(\bar{x},\bar{y})=0,}
\end{array}
\right.
\]
with $\bar{\alpha}\geq0$, $\bar{\alpha}^\top G(\bar{x})=0$, $\overline{\tilde{\gamma}}_{\eta}=0$, and $\nabla_{y}g_{\nu}(\bar
{x},\bar{y})\bar{\beta}=0$.

It remains to prove that $\overline{\tilde{\gamma}}_{i}\nabla_{y}g_{i}(\bar
{x},\bar{y})\bar{\beta}\leq0$ for $i\in\theta.$ Let us assume that for
some $i\in\theta$,
\[
\overline{\tilde{\gamma}}_{i}>0 \;\; \mbox{ and }\;\; \nabla_{y}g_{i} (\bar{x},\bar{y})\bar{\beta}>0.
\]
From (\ref{S8}) and the fact that
$\overline{\tilde{\gamma}}_{i}>0,$ $i\in\mathrm{supp}\gamma^{k}$ and $g_{i}(x^{k},y^{k})=0$ for all $k$
large enough, so that
\begin{equation*}
\nabla_{y}g_{i}(x^{k},y^{k})\beta^{k}=\left\{
\begin{array}{lll}
-\mu_{i}^{k} & \, \mbox{ if } i\in\mathrm{supp}\mu^{k},\medskip\\
0 & \, \mbox{ otherwise.}
\end{array}
\right.
\end{equation*}

Thus,
$
\nabla_{y}g_{i}(\bar{x},\bar{y})\bar{\beta} \leq 0,
$
which is a contradiction to our assumption that $\nabla
_{y}g_{i}(\bar{x},\bar{y})\bar{\beta}>0.$ Similarly, if $\overline
{\tilde{\gamma}}_{i}<0$ and $\nabla_{y}g_{i}(\bar{x},\bar{y})\bar{\beta
}<0.$ Thus, $\tilde{\gamma}^{k}<0$ for all $k$ large enough and by (\ref{S8}),
$i\in\mathrm{supp}\delta^{k}\backslash\eta$ for all $k $ large enough with
$i\notin\mathrm{supp}\mu^{k}$. Hence, for all $k$ large enough,
\[
\nabla_{y}g_{i}(x^{k},y^{k})\frac{\beta^{k}}{\left\Vert \chi
_{k}\right\Vert }=-\frac{\delta_{i}^{k}}{\left\Vert \chi_{k}\right\Vert }
g_{i}(x^{k},y^{k})
>0,
\]
and the limit leads to the inequality $\nabla_{y}g_{i}(\bar{x},\bar{y})\bar{\beta}\geq0$, which is
a contradiction. Hence, we have  the inequality $\overline{\tilde{\gamma}}_{i}(\nabla_{y}g_{i}(\bar
{x},\bar{y})\bar{\beta})\leq0$ for $i\in\theta$ and the conclusion
follows. \hfill \qed
\end{proof}
To illustrate this result, we apply it to the problem from Example
\ref{ex4}.
\begin{example}
Observe that for the optimal solution $\bar{x}=-1$ of
Example \ref{ex4}, $\mathcal{S}_{p}(\bar{x})$ is (nonempty)
compact and the conditions $(A_{1}^{m})$ and $(A_{2}^{m})$ are trivially satisfied at all $(\bar{x},y,u)\in \mathrm{gph}\,\mathcal{S}_{p}$. Observe also that $x^k=-1$ and $x^k=\sqrt[]{t_k} $  are the stationary points of  \eqref{RMPCCt} for $t:=t_k$. {{Evidently, property}} \eqref {e} is satisfied for  $x^k=-1$. In fact, we have for any sequence $(x_{k})_k$
in $\left[  -1,0\right[  $ such that $x_{k}\rightarrow\bar{x}=-1$,
\[
\mathcal{S}^{t_{k}}_p(x_{k})=\left\{  (1,(u_{k},u_{k}-x_{k}))\left| \;0\leq u_{k}\leq
t_{k}\right.\right\} \;\; \mbox{ for all } \;\; k.
\]
Hence, giving $(u_{k})$ such that $0\leq u_{k}\leq t_{k}$ with  $t_{k}\rightarrow 0^+$, we get
\[
d((1,(u_{k},u_{k}-x_{k})), \; \mathcal{S}_{p}(\bar{x})) \,=\, \left\Vert
(1,(u_{k},u_{k}-x_{k})) - (1,\, (0,1))\right\Vert  \, \leq \, 2u_{k}+\left\vert x_{k}+1\right\vert \longrightarrow 0
\]
and thus $e\left(\mathcal{S}^{t_{k}}_p(x_{k}), \;\mathcal{S}_{p}(\bar{x})\right) \longrightarrow 0$ as $k \longrightarrow \infty$.

However, for the stationary point $x^k=\sqrt[]{t_k}$, the property \eqref {e} does not hold since
\begin{equation}\tag*{\mbox{\qed}}
    e\left(\mathcal{S}^{t_{k}}_p(\sqrt[]{t_k}),\; \mathcal{S}_{p}(0)\right)=\left\Vert \left(\frac
{t_{k}}{\sqrt[]{t_k}},\, \sqrt[]{t_k}, \,0\right)-(1,\, 0, \,0)\right\Vert =\left\vert \sqrt[]{t_k}-1\right\vert + \sqrt[]{t_k} \longrightarrow 1.
\end{equation}
\end{example}
Note that condition \eqref{e} is satisfied if the set-valued mapping $(t,x) \rightrightarrows \mathcal{S}^{t}_p(x)$ is H-upper semicontinuous at $(0,\bar{x})$; i.e.,
$
e(\mathcal{S}^{t}_p(x), \; \mathcal{S}^{0}_p(\bar{x})) \longrightarrow 0
$
as $(t,x)\longrightarrow (0,\bar{x})$,  taking into account the equality $\mathcal{S}^{0}_p(\bar{x})=\mathcal{S}_{p}(\bar{x})$.
{{The following example shows that the condition \eqref{e} can be satisfied in the absence of the H-upper semicontinuity of the set-valued mapping $(t,x) \rightrightarrows \mathcal{S}^{t}_p(x)$.}}
\begin{example}
{{Consider a scenario of problem \eqref{PBP} defined by
\begin{equation}
F(x,y):=\ln{x}-\left|  y \right| \text{ for }\, x\in ]0,1],\text{ \ } F(0,y):=-\left|  y \right|,\,\,  K(x):=\left]  -\infty,1\right]\text{ and \ \ }f(x,y):= -  x y . \label{e1}
\end{equation}
Then we obtain $\psi_{p}(0)=0$ while for $x\neq 0$, $\psi_{p}(x)=\ln{x}-1$ given that
\[ \mathcal{S}_p(x)=\left\{
\begin{array}
[c]{c}
\left\{ (1,\;x ) \right\}  \text{ \ if \ \ }x\neq0,\\[1ex]
\left\{  (0,\,0)\right\}  \text{ \ \  if \ \ }x=0.
\end{array}
\right.
\]
On the other hand, we get for $t>0$  sufficiently small,
\[ \mathcal{S}_{p}^{t}(x)=\left\{
\begin{array}
[c]{c}
\left\{\left(1-\dfrac{t}{  x },\;  x\right) \right\} \text{ \ if \ \ }x\neq0,\\[1ex]
\left\{  (0,0)\right\}  \;\;\; \; \; \text{ \ \ \ \ \ \  if \ \ }x=0.
\end{array}
\right.
\]}}${}$\\
{{Here, $\psi_{p}^{t}(0)=0$ while for $x \neq 0$,
$\psi_{p}^{t}(x)=\ln{x}-1+\dfrac{t}{  x }.$ Hence one can check that  $x^k=t_k $ is  the unique stationary point of  \eqref{RMPCCt} for $t:=t_k$ and the property \eqref {e}  holds since
\[
e\left(\mathcal{S}^{t_{k}}_p(x^k),\; \mathcal{S}_{p}(0)\right)=\left\Vert (0,t_k)-(0,0)\right\Vert =t_k \longrightarrow 0.
\] However for $u^k=\sqrt{t_k}$, we have $\lim_{k \rightarrow \infty} e(\mathcal{S}_{p}^{t_k}(u_k),\mathcal{S}_{p}(0))=1$. That is,  the set-valued mapping $(t,x) \rightrightarrows \mathcal{S}^{t}_p(x)$ is not H-upper semicontinuous at $(0,0)$.}}
\end{example}
The next result gives a framework that instead relies on semicontinuity properties of the set-valued mapping $\mathcal{D}$ and $\mathcal{D}^t$ to ensure that condition \eqref{e} is satisfied.
\begin{proposition}
Let $(t_{k})\downarrow0$ and let $(x^{k})\underset{k\rightarrow\infty}{\longrightarrow}\bar{x}$ such that $\mathcal{D}(\bar{x})$ is
nonempty and compact. Assume that the set-valued mapping $x \rightrightarrows \mathcal{D}(x)$ $($resp. $(t,x) \rightrightarrows \mathcal{D}^{t}(x)$$)$ is H-lower semicontinuous at any $(\bar{x}, y, u)\in\mbox{gph}\,\mathcal{D}$ $($resp. H-upper semicontinuous at $(0^+, \bar{x}))$. Then the property \eqref{e} holds.
\end{proposition}
\begin{proof}
Assuming that the statement is not true, that is, (\ref{e}) does not hold, then  there exist $\delta >0$ and a sequence $(y^{k},u^{k})_{k}$ such that $(y^{k},u^{k})\in\mathcal{S}^{t_{k}}_p(x^{k})$ and
\begin{equation}\label{not}
d((y^{k},u^{k}),\;\mathcal{S}_p(\bar{x}))\geq \delta \;\mbox{ for all }\; k.
\end{equation}
Therefore, since $(y^{k},u^{k})\in \mathcal{D}^{t_{k}}(x^{k})$, it holds that
\[
d((y^{k},u^{k}),\;\mathcal{D}(\bar{x}))\leq e(\mathcal{D}^{t_{k}}(x^{k}
),\mathcal{D}(\bar{x}))
\]
  with $e(\mathcal{D}^{t_{k}}(x^{k}),\mathcal{D}(\bar{x}))\longrightarrow 0$ as $k\longrightarrow\infty$,  since the set-valued mapping $(t,x) \rightrightarrows \mathcal{D}^{t}(x)$ is H-upper semicontinuous at $(0^+,\bar{x})$. Hence, there
exists a sequence  $(z^{k},w^{k})$ in $\mathcal{D}(\bar{x})$ such that
\[
\left\Vert y^{k}-z^{k}\right\Vert \longrightarrow0 \;\, \text{ and }\;\, \left\Vert
u^{k}-w^{k}\right\Vert \longrightarrow 0
\]
and due to the compactness of the set $\mathcal{D}(\bar{x})$, the sequence $(z^{k}, w^{k})$ (up to a subsequence) converges to a point $(\bar{y},\bar
{u})\in\mathcal{D}(\bar{x})$ and so does the subsequence $(y^{k},u^{k})_k$.

Let us prove now that $(\bar{y},\bar{u})\in\mathcal{S}_{p}(\bar{x})$. Indeed,
from assertion (iii) in Proposition \ref{lem}, we get
\[
F(x^{k},y^{k})-\psi_{p}(x^{k})\geq F(x^{k},y^{k})-\psi_p^{t_{k}}
(x^{k})\geq0
\]
as $(y^{k},u^{k})\in\mathcal{S}^{t_{k}}_p(x^{k})$. On the other
hand, since $F(\cdot,\cdot)$ is continuous and $\mathcal{D}(\cdot)$ is H-lower
semicontinuous, the function $\psi_{p}(\cdot)$ is lower semicontinuous at
$\bar{x}$. Hence,
\[
F(\bar{x},\bar{y})-\psi_{p}(\bar{x})\geq F(\bar{x},\bar{y})-\underset
{k\rightarrow\infty}{\text{ }\lim\inf}\psi_{p}(x^{k})=\underset
{k\rightarrow\infty}{\text{ }\lim\sup}(F(x^{k},y^{k})-\psi_{p}(x^{k}
))\geq0.
\]
Consequently, $(\bar{y},\bar{u})\in\mathcal{S}_{p}(\bar{x})$ and from the following equality
\[
d((y^{k},u^{k}),\;\mathcal{S}_{p}(\bar{x}))\leq\left\Vert (y^{k},u^{k})-(\bar
{y},\bar{u})\right\Vert
\]
we get $\underset{k\longrightarrow\infty}{\lim}d((y^{k},u^{k}), \;\mathcal{S}_{p}
(\bar{x}))=0$, which is a contradiction to (\ref{not}).  \hfill \qed
\end{proof}

In Theorem \ref{ConvergenceResult}, the requirement that the sequence $(x^{k})_k$ be such that each $x^k$ is a stationary point of problem
\eqref{RMPCCt} for $t:=t^k$ assumes the existence of a vector $(y^{t},u^{t}, \alpha^{t}, \beta^{t}, \gamma^{t}, \mu^{t},\delta^{t})$ such that the optimality conditions  \eqref{Er0}--\eqref{Er5} are satisfied. A priori, it is not clear
whether such a vector actually exists. {{According to Theorem \ref{Optimality conditions for RMPCC(t)}, such a vector can exist if the point $x^{k}$ is an upper-level regular local optimal solution for \eqref{RMPCCt} for $t:=t^k$}}, the set-valued mapping $\mathcal{S}^{t_k}_p$ is inner semicontinuous at $(x^k, y^k, u^k)$ for some $(y^k, u^k)\in \mathcal{S}^{t_k}_p (x^k)$, and the CQ \eqref{CQ1} holds at $(x^k, y^k, u^k)$.  The question is, how can we ensure that these three conditions are automatically satisfied as our sequence gets close to the point of interest.  Theorem \ref{Continuity of set-valued maps}(iii) has already given some conditions  ensuring that we can find a neighborhood $U\times V$ of some reference point $(\bar{x},(\bar{y},\bar{u}))$ such that for all $t\downarrow0$, $\mathcal{S}_{p}^{t}$ is inner semicontinuous at all points $(x,y,u)$
with $x\in U$ and $(y,u)\in V\cap\mathcal{S}_{p}^{t}(x)$.

Theorem \ref{TheoremFinal} below, which is a variant of \cite[Proposition 2.2]{Qi}, establishes  sufficient conditions for conditions \eqref{UMFCQ} and \eqref{CQ1} to hold at all points in a neighborhood of our reference point. Since the challenging part of this result relies on the satisfaction of the qualification conditions $(A_{1}^{m})$ and $(A_{2}^{m})$ at our reference point, we first state the stability of $(A_{1}^{m})$ and $(A_{2}^{m})$, in the sense that if they hold at a given point, they also hold at all points in some neighborhood of the point.

\begin{proposition}\label{lim}
Assume that conditions $(A_{1}^{m})$ and $(A_{2}^{m})$ hold at $(\bar{x}
,\bar{y},\bar{u})$. Then there exists a  neighborhood $U\times V$ of $(\bar
{x},(\bar{y},\bar{u}))$ such that $(A_{1}^{m})$ and $(A_{2}^{m})$ hold at any { $(x,y,u)$ with $x\in U$ and  $(y,u)\in V\cap \mathcal{D}(x)$}.
\end{proposition}

\begin{theorem}\label{TheoremFinal}
Let $\bar{x}$ be upper-level regular  and suppose that  $(A_{1}^{m})$ and
$(A_{2}^{m})$ hold at $(\bar{x},\bar{y},\bar{u})$ with $(\bar{y},\bar{u}
)\in\mathcal{D}(\bar{x})$. Then there exist neighborhoods $U$ and $V$ of
$\bar{x}$ and $(\bar{y},\bar{u})$, respectively, such that for all $t$ close to $0^+$, the
CQs \eqref{UMFCQ} and \eqref{CQ1} for \eqref{RMPCCt} are
satisfied at all points $(x,y,u)$ with $x\in U\cap X$ and $(y,u)\in V\cap
\mathcal{D}^{t}(x)$.
\end{theorem}
The proofs of these two results are given in Appendix \ref{Proof of Proposition lim} and Appendix  \ref{Proof of Theorem TheoremFinal},
respectively.

\section{Practical implementation and numerical experiments}\label{Sec:Numerical}
Based on the convergence results from Sections \ref{Computing global and local optimal solutions} and \ref{Computing C-stationary points}, we can consider two possible ways to solve problem \eqref{MPCC} by means of the Scholtes relaxation. For the first option, we need a solver for the minmax problem \eqref{RMPCCt} for a fixed value of $t>0$. Given the complicated nature of the coupled inner feasible set of the problem, this is a very challenging problem, and we are not aware of a handy tool to efficiently solve it.
However, we can easily have access to tractable
solvers for systems of equations. Hence, we are going to focus our attention here at computing C-stationary points for problem \eqref{MPCC} based on a completely detailed form of \eqref{Er0}--\eqref{Er5}.
Clearly, it suffices to provide a detailed form for condition \eqref{Er0}.
From the definition of {$\mathcal{S}^{t}_p$}, we have the
equivalence
\begin{equation}\label{ULVF}
(x^t, y^t, u^t)\in\mathrm{gph}\mathcal{S}^{t}_p \quad\Longleftrightarrow
\quad\left[ (y^t, u^t)\in\mathcal{D}^{t}(x^t) \;\; \mbox{ and } \;\; F(x^t, y^t) {-}
\psi^{t}_p(x^t)\geq0\right],
\end{equation}
which corresponds to the optimal value reformulation of the underlying
parametric optimization problem. Considering the challenge in dealing with the
optimal value function in the process of solving bilevel optimization problems
(see, e.g., \cite{FischerZemkohoZhou,FliegeTinZemkoho,TinZemkohoLevenberg,ZemkohoZhouTheoretical}), we will instead consider the \emph{KKT-type
approach} here, which ensures that if $(x^t, y^t,
u^t)\in\mathrm{gph}\mathcal{S}^{t}_p$, then we have
\begin{equation}\label{Toumana}
\nabla_{y, u}F(x^t, y^t) \in  N_{\mathcal{D}^{t}(x^t)}(y^t, u^t).
\end{equation}
\begin{figure}[htp]
    \centering
 	\includegraphics[height=6.0cm]{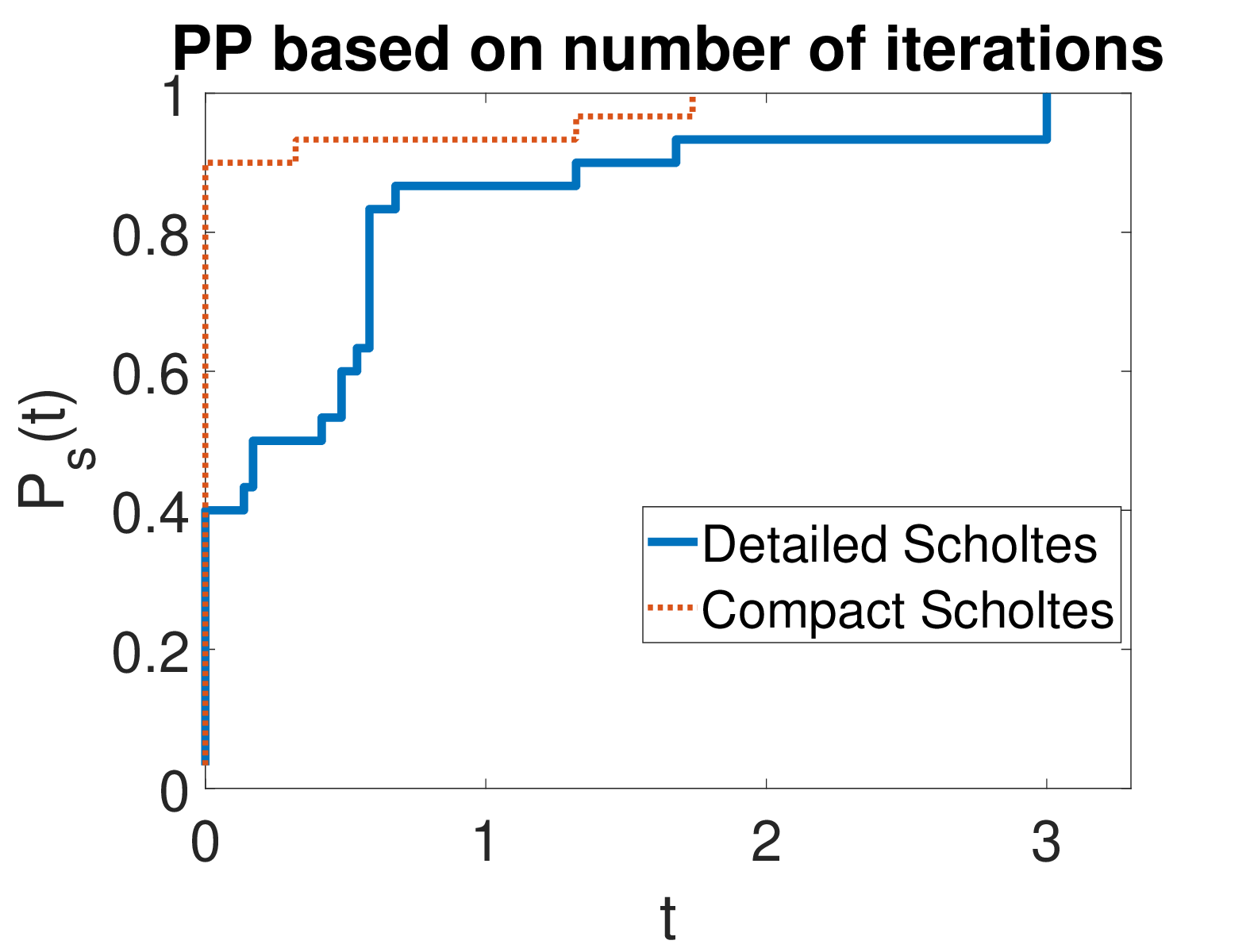}
       	\includegraphics[height=6.0cm]{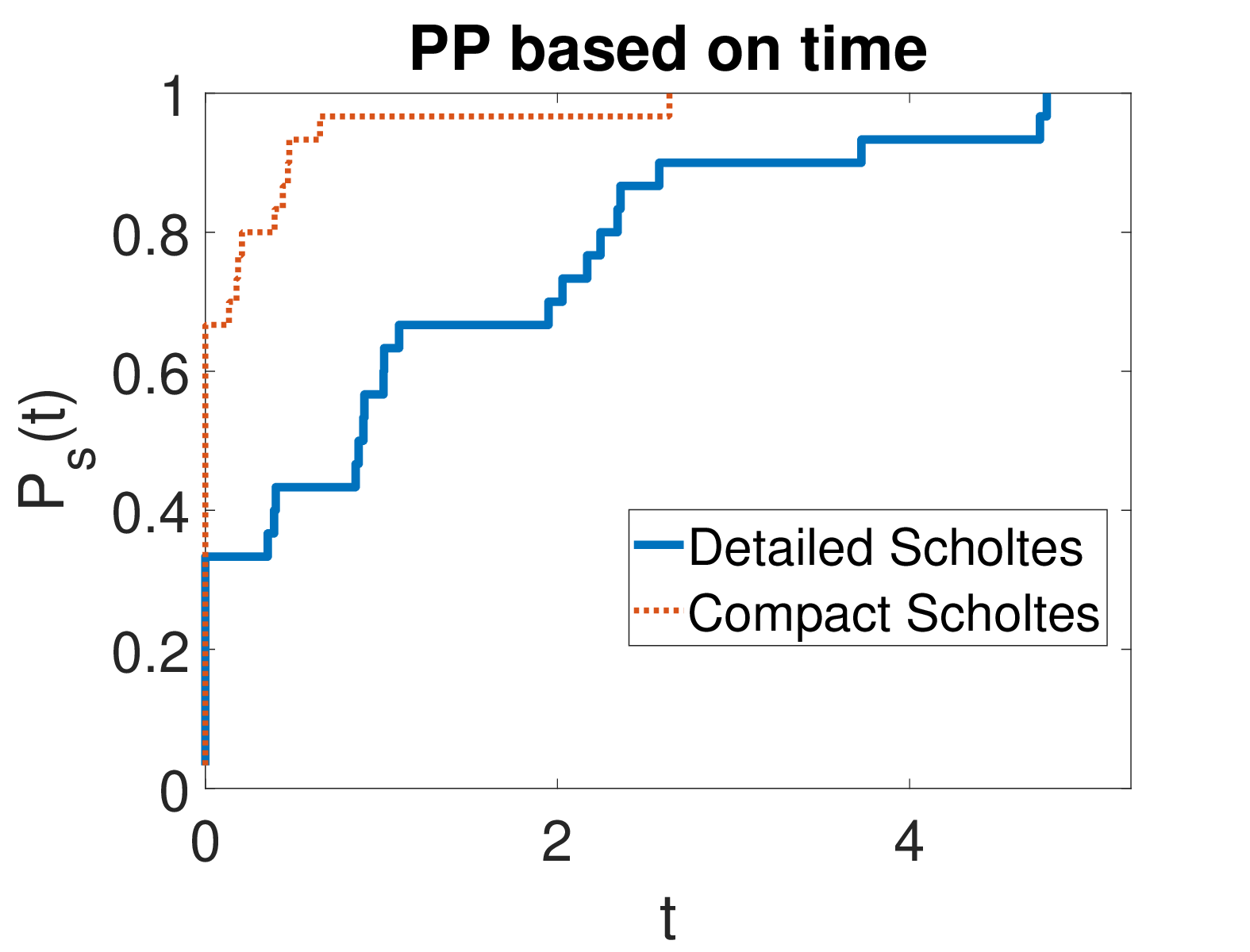}
        	\includegraphics[height=6.0cm]{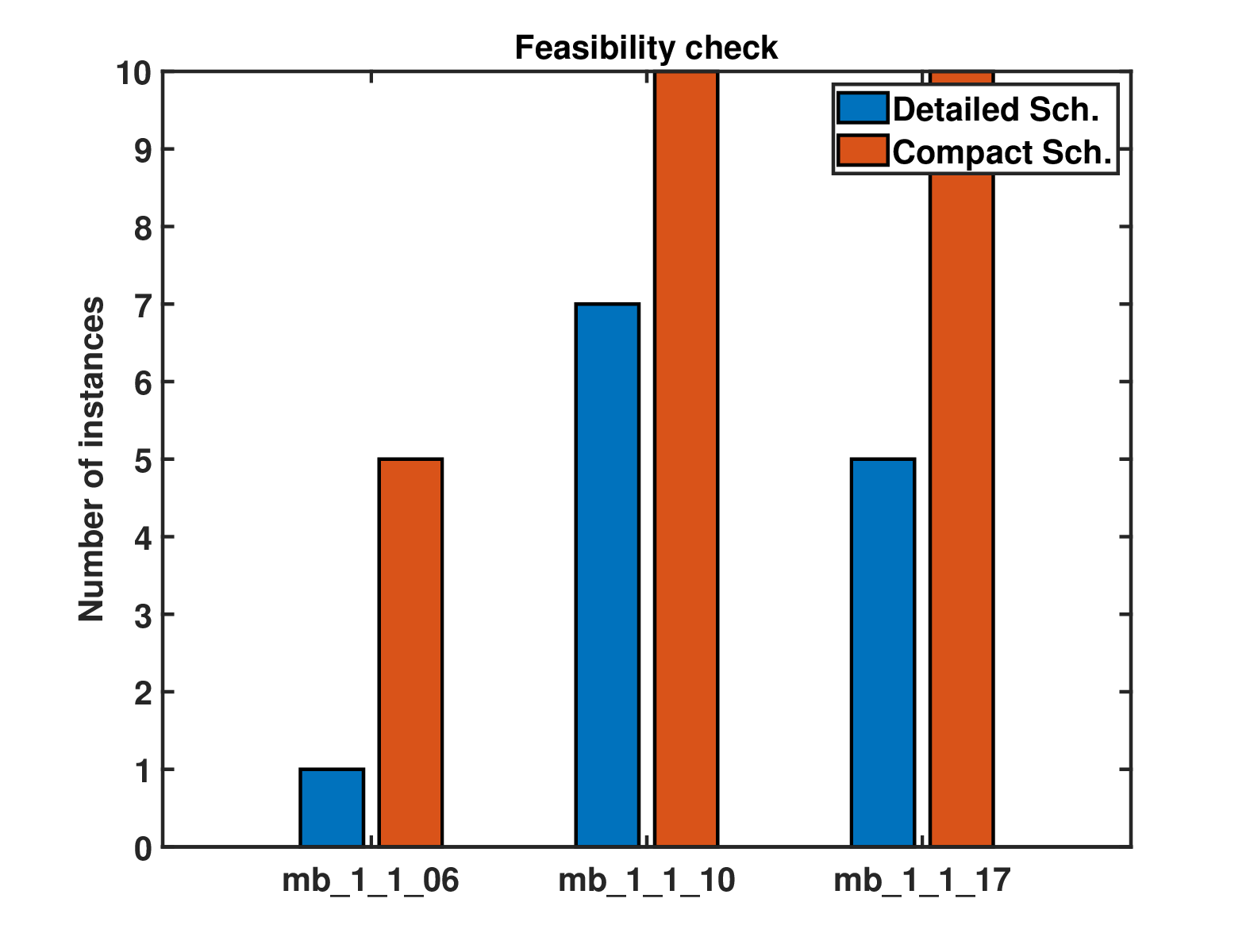}
      \includegraphics[height=6.0cm]{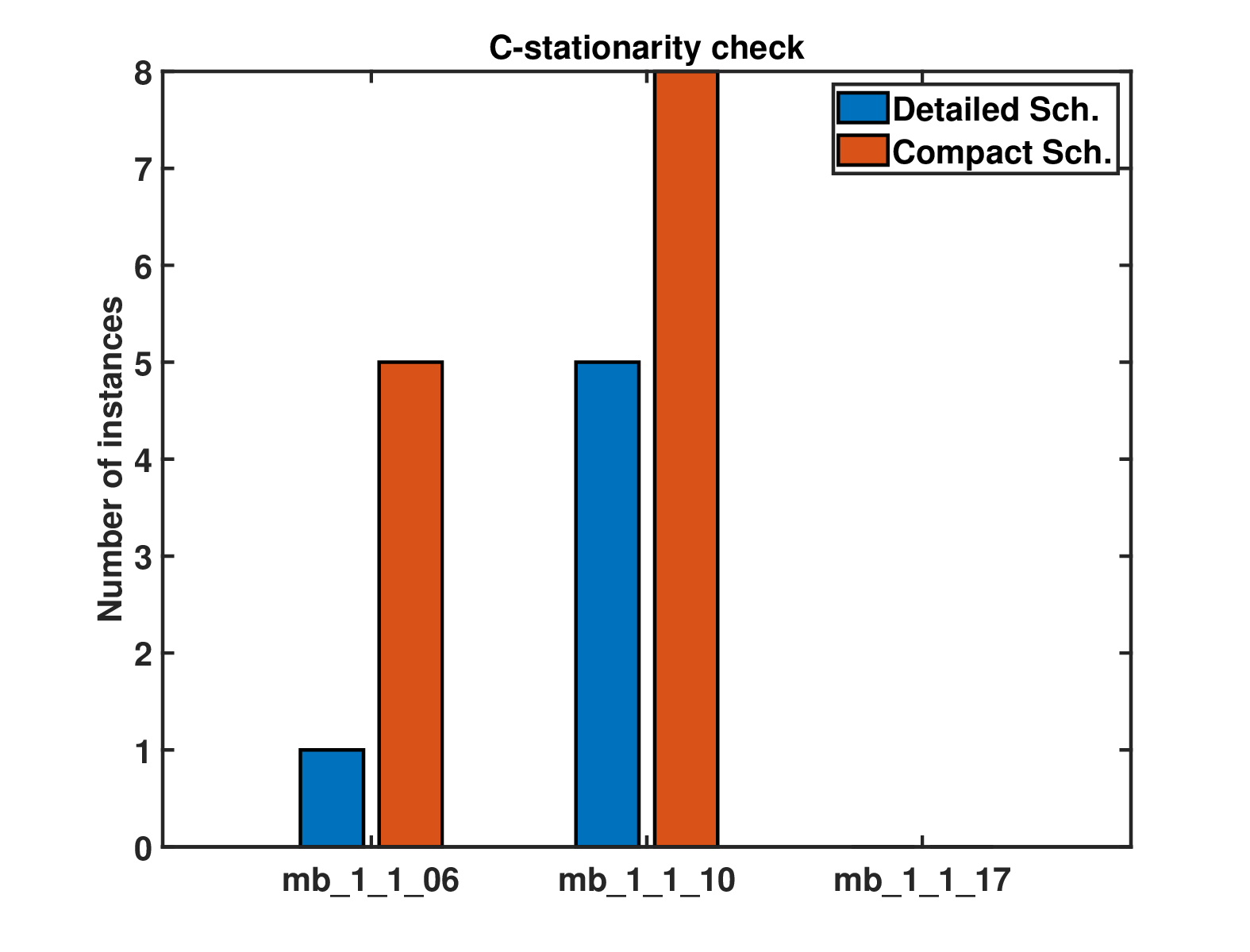}
      \caption{Performance profile (PP) and feasibility and C-stationary check for detailed versus compact form with examples satisfying the lower-level convexity and regularity conditions (Experiment I).}\label{fig1}
\end{figure}
If the point $(x^t, y^t, u^t)$ satisfies a certain CQ for the conditions defining the set $D^t(x^t)$ in \eqref{Dt}, then we can find some $(\beta^t, \delta^t, \gamma^t)$ such that the conditions \eqref{Er1-f1}--\eqref{Er1-f2} and  \eqref{Er1-f3}--\eqref{Er5} hold. This clearly means that if some form of sufficient condition is satisfied, then a point satisfying \eqref{Er1-f1}--\eqref{Er1-f2} and  \eqref{Er1-f3}--\eqref{Er5} could also satisfy condition \eqref{Er0}. CQs to ensure that a point satisfying the latter condition implies the fulfilment of \eqref{Er1-f1}--\eqref{Er1-f2} and  \eqref{Er1-f3}--\eqref{Er5} and sufficient conditions to guarantee the equivalence between the two is out of the scope of this paper and will be analyzed with more care in a separate piece of work.

The main point to take from these observations is that solving the system \eqref{Er1}--\eqref{Er5} can be a sensible proxy to compute stationary points for \eqref{RMPCCt} for each fixed $t>0$. This is what we are going to do here, just to give a flavour of the potential for the Scholtes relaxation method in the context of the pessimistic bilevel optimization problem. And numerical examples presented later in this section show that solving \eqref{Er1}--\eqref{Er5} as $t\downarrow 0$ presents a good potential in computing optimal solutions and/or C-stationary points for problem \eqref{PBP}. 

Clearly, the focus of the numerical calculations in this section will be to solve the system of optimality conditions in \eqref{Er1}--\eqref{Er5} by means of the well-known Fischer-Burmeister  function \cite{FischerSpecial1992}. Denoting by
$\zeta^t:=\left(x^t, y^t, u^t, \alpha^t, \beta^t, \gamma^t, \delta^t, \mu^t\right)$,
this can be reformulated into
\begin{equation}\label{MainEquation}
\Phi^{S}_t(\zeta^t) =0,
\end{equation}
which is a square $n + 2m + p + 3q +1$ by $n + 2m +p + 3q + 1$  system of
equations. The size of this system can be reduced by substituting
$
-\mu^{t} = \nabla_{y}g(x^{t}, y^{t})\beta^{t} + \delta^{t} g(x^{t}, y^{t})
$
from  \eqref{Er1-f2} into \eqref{Er1-f4}.  With this transformation, the corresponding version of \eqref{MainEquation} will be a square $n + 2m + p + 2q +1$ by $n + 2m +p + 2q + 1$  system of equations. As part of the analysis in this section, we will compare the behavior of Algorithm \ref{algorithm 1} in these two scenarios. In Step 1, we will solve the corresponding version of the system in \eqref{MainEquation}. However, to be able to use an off-the-shelf solver for smooth equations, we consider the following smooth approximation $\phi^\epsilon :\mathbb{R}^2 \rightarrow \mathbb{R}$ of the Fischer-Burmeister function to deal with the corresponding version of the  complementarity conditions \eqref{Er011}--\eqref{Er5}:
\begin{equation}
\phi^\epsilon (a, b):=\sqrt{a^2 + b^2 + 2\epsilon} - a +b.
\end{equation}
The smoothing parameter $\epsilon>0$ helps to guarantee the differentiability of the function at the point $(a,b)=(0,0)$. Note that it is well known that
\begin{equation}
\phi^\epsilon (a, b) = 0  \;\;\; \Longleftrightarrow \;\;\; a >0, \;\;\; b<0, \;\;\; ab=-\epsilon.
\end{equation}
Hence, the corresponding detailed and compact smooth approximation of equation \eqref{MainEquation} become
\begin{align}\label{Detailed}
    \Phi^{S_{{d}}}_{t}(\zeta) : = \left( \begin{array}{c}
     \nabla_{x}F(x^{t}, y^{t}) + \nabla G(x^{t})^{\top}\alpha^{t}+ \nabla_{x}\mathcal{L}(x^{t}, y^{t},
u^{t})^{\top}\beta^{t}+ \nabla_{x}g(x^{t}, y^{t})^{\top}\left( \delta^{t} u^{t} - \gamma^{t}\right)\\
{
\nabla_{y}F(x^{t}, y^{t}) + \nabla_{y}\mathcal{L}(x^{t}, y^{t}, u^{t})^{\top}\beta^{t}+ \nabla_{y}g(x^{t}
y^{t})^{\top}\left( \delta^{t} u^{t} - \gamma^{t}\right) }\\
-{\nabla_{y}g(x^{t}, y^{t})\beta^{t}+ \mu^{t}+ \delta^{t} g(x^{t}, y^{t})}\\
\phi^\epsilon(\alpha_{j}^{t}, G_{j}(x^{t}))_{j=1, \ldots p}\\
\phi^\epsilon(\gamma_{i}^{t}, g_{i}(x^{t}, y^{t}))_{i=1, \ldots q}\\
\phi^\epsilon(\mu_{i}^{t}, -u_{i}^{t})_{i=1, \ldots q}\\
\phi^\epsilon(\delta_{i}^{t}, -u_{i}^{t}g_{i}(x^{t}, y^{t})-t)_{i=1, \ldots q}\\
{\mathcal{L}(x^{t}, y^{t}, u^{t})}
\end{array}
    \right) = 0
\end{align}
and
\begin{align}\label{Compact}
    \Phi^{S_{{c}}}_{t}(\zeta) : = \left( \begin{array}{c}
     \nabla_{x}F(x^{t}, y^{t}) + \nabla G(x^{t})^{\top}\alpha^{t}+ \nabla_{x}\mathcal{L}(x^{t}, y^{t},
u^{t})^{\top}\beta^{t}+ \nabla_{x}g(x^{t}, y^{t})^{\top}\left( \delta^{t} u^{t} - \gamma^{t}\right)\\
{
\nabla_{y}F(x^{t}, y^{t}) + \nabla_{y}\mathcal{L}(x^{t}, y^{t}, u^{t})^{\top}\beta^{t}+ \nabla_{y}g(x^{t},
y^{t})^{\top}\left( \delta^{t} u^{t} - \gamma^{t}\right) }\\
\phi^\epsilon(\alpha_j^{t}, G_j(x^{t}))_{j=1, \ldots p}\\
\phi^\epsilon(\gamma_i^{t}, g_i(x^{t}, y^{t}))_{i=1, \ldots q}\\
\phi^\epsilon(u_i^{t}, \nabla_{y}g_{i}(x^{t}, y^{t})\beta^{t} -\delta_{i}^{t} g_{i}(x^{t}, y^{t}) )_{i=1, \ldots q}\\
\phi^\epsilon(\delta_i^{t}, -u_{i}^{t}g_i(x^{t}, y^{t})-t)_{i=1, \ldots q}\\
{\mathcal{L}(x^{t}, y^{t}, u^{t})}
\end{array}
    \right) = 0,
\end{align}
respectively. Here, $\Phi_{t}^{S_d}$ and $\Phi_{t}^{S_c}$ denote the detailed and compact form of the necessary optimality conditions from the Scholtes relaxation, respectively. 

%

We now illustrate the numerical performance of Algorithm \ref{algorithm 1} using MATLAB (R2022a) on some test problems from \cite{Mitsos2006}. We choose the default starting point $\zeta^{o} = (x^{o},y^{o},u^{o},\alpha^{o},\beta^{o},\delta^{o},\gamma^{o},\mu^{o})$ in the following way: the pair $(x^o,y^o)$ are generated randomly while we set $u^{o} = (|g_{1}(x^{o},y^{o})|, \dots, |g_{p}(x^{o},y^{o})|),$ $\alpha^{o}=(|G_{1}(x^{o})|, \dots, |G_{m}(x^{o})|)$, $\beta^{o}=y^{o}$, $u^{o}=\delta^{o}=\gamma^{o}=\mu^{o}$. For every subsequent iteration, the starting point for Step 1 is the solution obtained from the previous iteration.  The algorithm is stopped if the stopping criterion $\|\Phi_{t}^{S}(\zeta^{k})\|<10^{-6}$ is reached or the maximum iteration limit of $10^4$ is exceeded. In the case when this is not satisfied, we update the iterate of the algorithm by setting $t^{k+1} = \theta t^{k}$, where $\theta = 0.05$.
\begin{figure}[htp]
    \centering
 	\includegraphics[height=6.0cm]{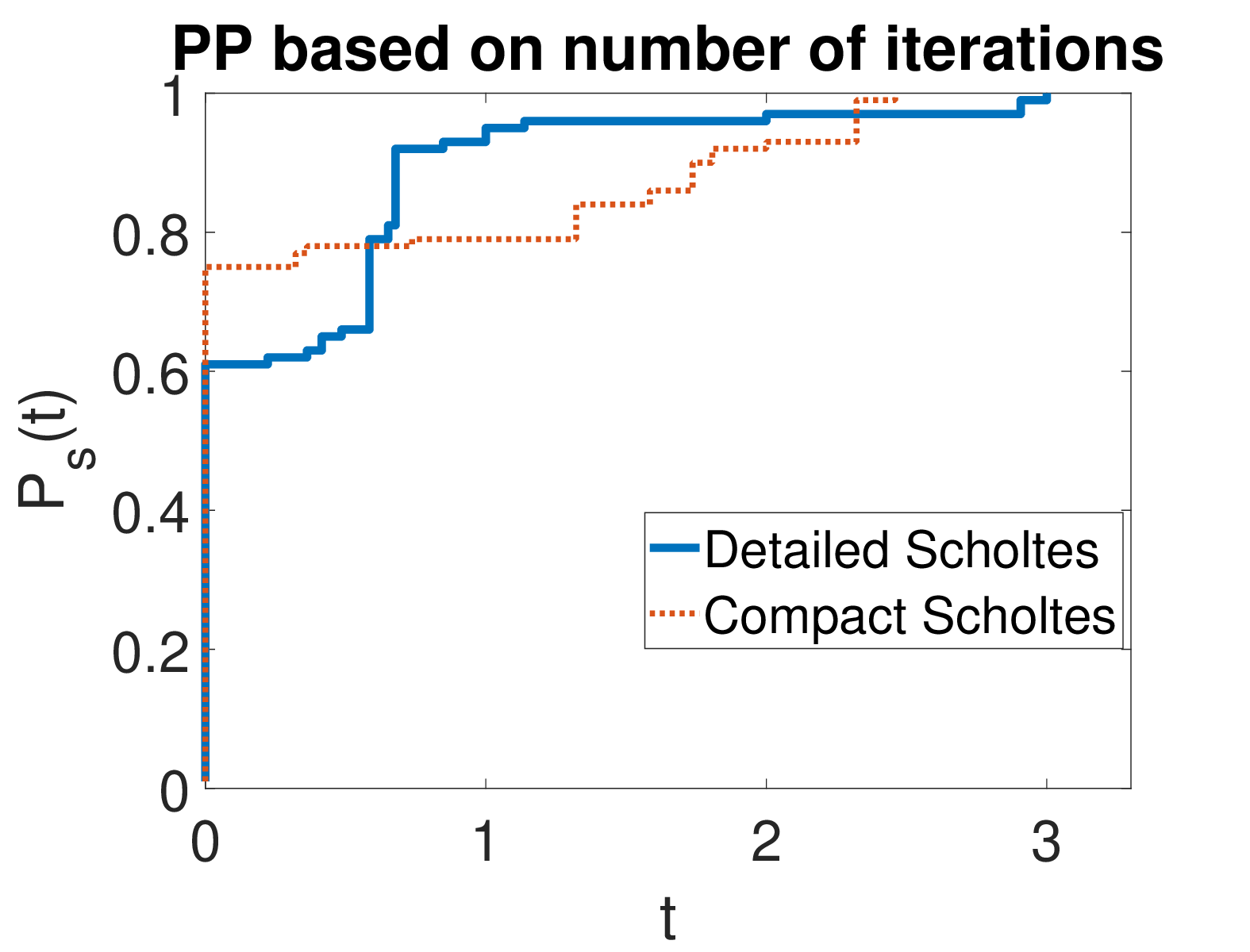}
       	\includegraphics[height=6.0cm]{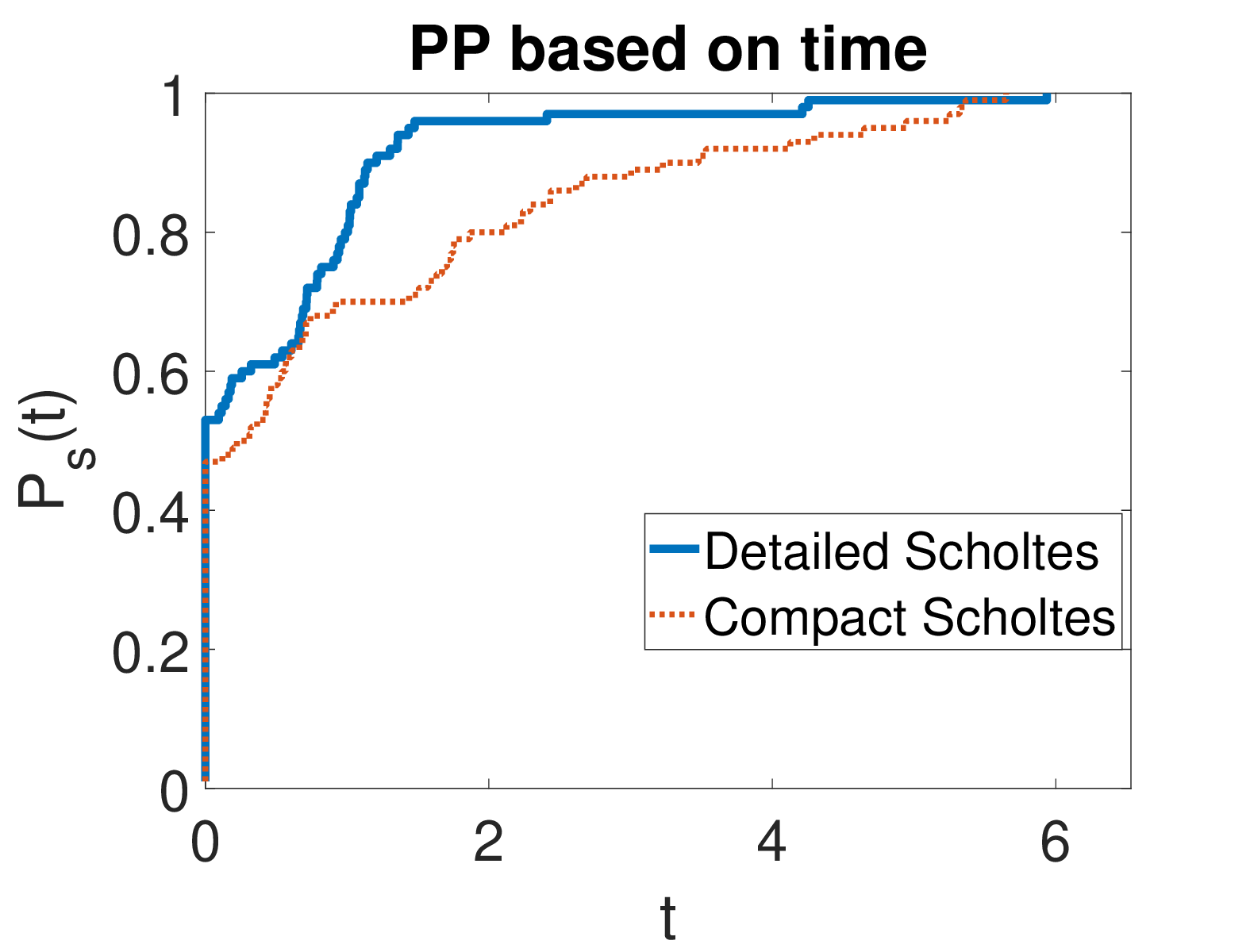}
        	\includegraphics[height=6.0cm]{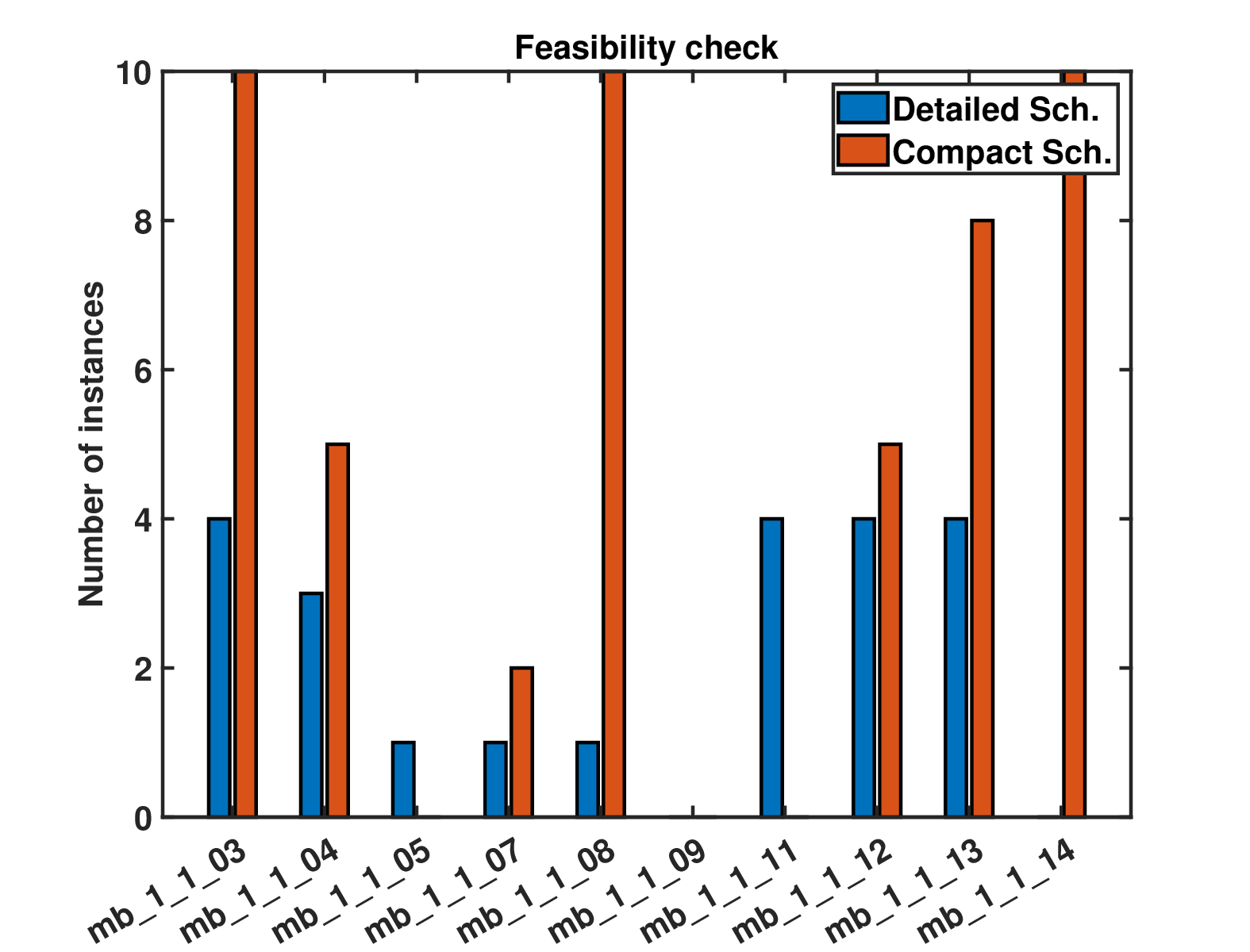}
      \includegraphics[height=6.0cm]{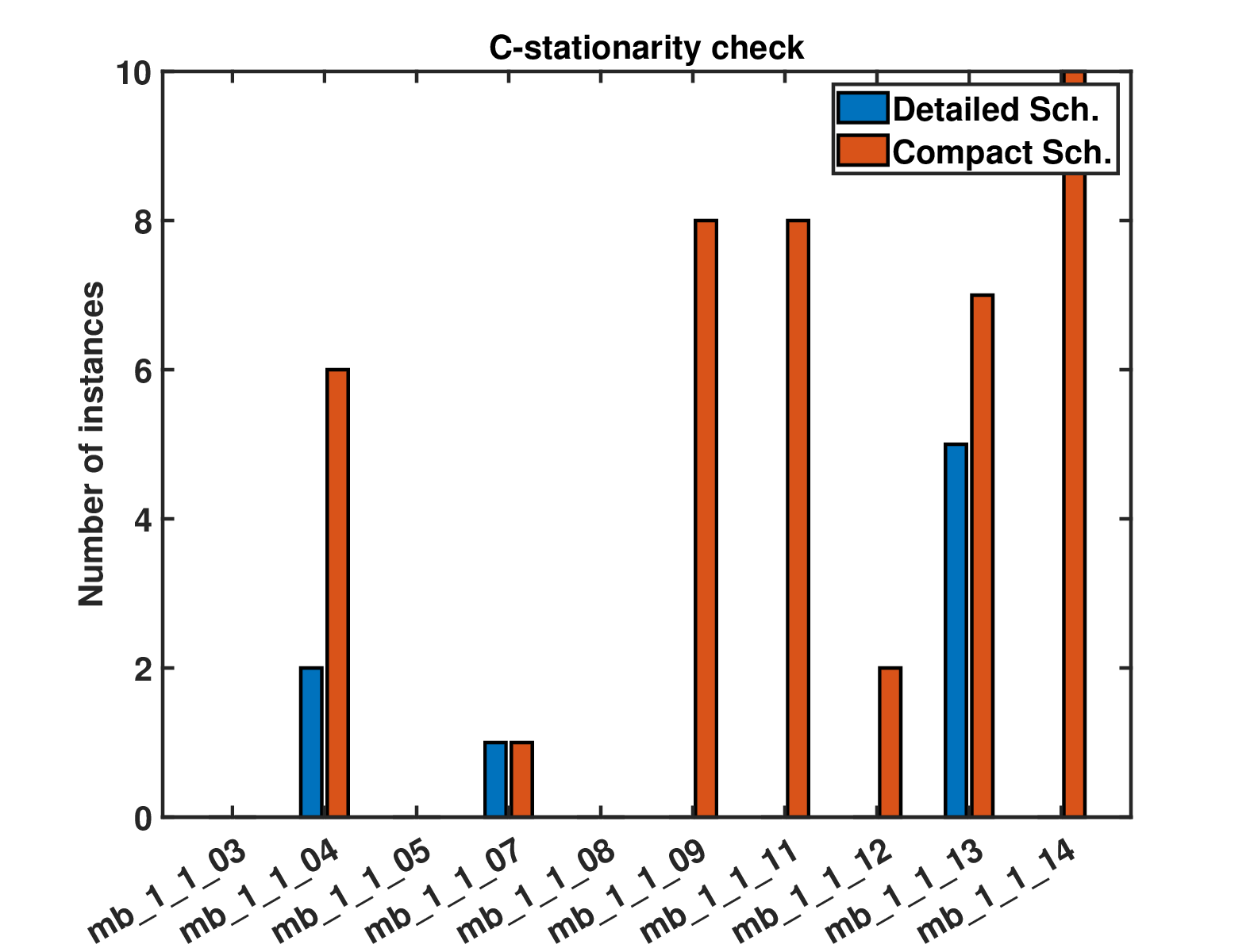}
      \caption{Performance profile of detailed versus compact form for with examples not satisfying convexity and regularity conditions (Experiment II).}\label{fig2}
\end{figure}

The performance of the algorithm is measured using the profile by Dolan and M\'{o}re \cite{DM}, which has widely been used to compare  numerical methods. We denote by $t_{i,s}$ the metric of comparison for a solver $s \in \mathcal{S}$ to solve problem $i \in \mathcal{I}$ and defined the performance ratio by
\begin{equation}
   { r_{i,s} = \frac{t_{i,s}}{\min\{t_{i,s'}: s' \in \mathcal{S}\}}, \; \; \forall s\in \mathcal{S},\;\, i \in \mathcal{I}, }\nonumber
\end{equation}
where $\mathcal{S}$ is the set of all solvers and $\mathcal{I}$ is the set of all problems used in the experiment. Note that $r_{i,s}$ is the ratio of the performance of solver $s \in \mathcal{S}$ to solve problem $i \in \mathcal{I}$ compared to the best performance of any other solver in $\mathcal{S}$ to solve $i$. The cumulative function $P_{s}: [1,\infty) \to [0,1]$ of the current performance profile index associated with solver $s$ is defined by
\begin{equation}
    P_{s}(t) := \frac{|\{i \in \mathcal{I} | r_{i,s} \leq t\}|}{|\mathcal{I}|}, \label{pp}
\end{equation}
where $|\mathcal{I}|$ is the cardinality of $\mathcal{I}.$  By \eqref{pp}, the performance profile index is counting the number of problems for which the performance ratio of the solver $s$ is better than $\tau$. Note that $\omega_{s}$ is a non-decreasing function, where $P_{s}(1) $ represents the fraction of problems for which solver $s \in \mathcal{S}$ shows the best performance.

From \cite{Mitsos2006}, we consider the problems whose optimal pessimistic solutions are described in \cite{WiesemannTsoukalasKleniatiRustem2013} that  consists of 18 test examples. Out of these examples, we identify three examples, namely, \texttt{mb\_1\_1\_06}, \texttt{mb\_1\_1\_10}, and \texttt{mb\_1\_1\_17}, where the lower-level problem is convex and has linear lower-level constraints, both w.r.t. $y$. This ensures that this set of examples satisfy the basic assumptions ensuring from problem \eqref{MPCC} is well-defined. Also, we identify another set of examples whose upper-level constraint function is independent from $y$, thus conforming with the type of pessimistic problem described in \eqref{PBP}, but with problem \eqref{MPCC} not well-defined, as the lower-level problem is either nonconvex or does not necessarily satisfy the MFCQ w.r.t. $y$. For each of  these 10 examples, \texttt{mb\_1\_1\_03}, \texttt{mb\_1\_1\_04}, \texttt{mb\_1\_1\_05}, \texttt{mb\_1\_1\_07}, \texttt{mb\_1\_1\_08}, \texttt{mb\_1\_1\_09}, \texttt{mb\_1\_1\_11}, \texttt{mb\_1\_1\_12}, \texttt{mb\_1\_1\_13}, and \texttt{mb\_1\_1\_14} of the form  \eqref{PBP}, we know a true optimal solution. We divide the numerical computations into Experiment I (with the three examples for which the KKT reformulation problem \eqref{MPCC} is well-defined) and  Experiment II for the remaining seven problems.   Each example is tested with 10 randomly generated starting points as described above and the result in each instance is recorded for comparison.
\begin{table}[h!]
    \centering
    \begin{tabular}{c|ccccc}
        Problem & status & $F_{opt}$ (known) & $F_{pes}$ (known) & $F(x,y)$ (detailed) & $F(x,y)$ (compact)  \\
        \hline
           \texttt{mb\_1\_1\_03} &N & 0.5 & 0.5 & 0.5 (6) & 0.5 (8)   \\
   \texttt{mb\_1\_1\_04} &N & -0.8 & 0.5 & 0.5 (8) & 0.5 (7)  \\
     \texttt{mb\_1\_1\_05} & N & 0 & 0 & 0.02 (4) & 0.02 (6)  \\
      \texttt{mb\_1\_1\_06} &Y & -1 & 0 & 0 (6) & 0 (6)   \\
\texttt{mb\_1\_1\_07} & N & 0.25 & 0.2507 & 0.062 (10) & 0.063 (10)  \\
   \texttt{mb\_1\_1\_08} &N & 0 & 0 & -1.0217 (10) & -1.0161 (9)  \\
     \texttt{mb\_1\_1\_09} &N & -2 & -2 & -2 (8) & -2 (8) \\
      \texttt{mb\_1\_1\_10} & Y & 0.1875 & 0.1875 & 0.1875 (8) & 0.1875 (8)  \\
\texttt{mb\_1\_1\_11} &N & 0.25 & 0.251 & -0.03 (5) & -0.03 (5) \\
   \texttt{mb\_1\_1\_12} &N & -0.258 & 0 & -0.2581 (7) & -0.258 (9)  \\
     \texttt{mb\_1\_1\_13} &N & 0.3125 & 0.3135 & 0.3095 (10) & 0.3012 (10) \\
  \texttt{mb\_1\_1\_14}& N & 0.2095 & 0.2095 & 0.2095 (5) & 0.2095 (5) \\
\texttt{mb\_1\_1\_17} &Y & -1.7550  & -0.2929 & -1.7550 (7) / -0.2929 (2) & -1.7550 (7) /-0.2929 (2)  \\
\hline
    \end{tabular}
    \vspace{4pt}
    \caption{Comparison of objective function value $F(x,y)$ obtained by algorithm with the known $F(x,y)$ in \cite{WiesemannTsoukalasKleniatiRustem2013}. The number in the parenthesis indicates the number of instances the optimal value is obtained out of 10 instances.}
    \label{tab:my_label}
\end{table}

In Table \ref{tab:my_label}, the first column contains the names of the example as labelled in \cite{Mitsos2006} and \cite{WiesemannTsoukalasKleniatiRustem2013}. The second column indicates if the problem satisfies both lower-level convexity and linearity of the constraints  (Y) or otherwise (N). The column $F_{opt}$ (known) contains the known upper-level objective function value of the optimal solution for the optimistic version of the problem as obtained in   \cite{WiesemannTsoukalasKleniatiRustem2013}, column $F_{pes}$ (known) contains the upper-level objective function value of the optimal solution for the pessimistic version of the problem as obtained in \cite{WiesemannTsoukalasKleniatiRustem2013}. The column $F(x,y)$ (detailed) is the obtained/computed upper-level objective function value by the Scholtes Algorithm with \eqref{Detailed} and the column $F(x,y)$ (compact) is the computed upper-level objective function value by the Scholtes Algorithm using \eqref{Compact}. We also compare the performance of the methods in terms of the number of outer iterations, time of execution, number of inner iterations (i.e., the number of iterations for MATLAB's \texttt{fsolve} to solve the system), the experimental order of convergence (EOC), and the number of instances for which the algorithm obtains a C-stationary point. 
Note that the EOC is defined by
$$EOC  = \max \left\{ \frac{\log\|\Phi_{t^k}^{\mathcal{S}}(\zeta^{K-1})\|}{\log\|\Phi_{t_k}^{\mathcal{S}}(\zeta^{K-2})\|}, \frac{\log\|\Phi_{t_k}^{\mathcal{S}}(\zeta^{K})\|}{\log\|\Phi_{t_k}^{\mathcal{S}}(\zeta^{K-1})\|} \right\}.$$
In order to check the feasibility of the point obtained by the algorithm, from \eqref{KKT system}, we checked that the point obtained by the algorithm satisfied $u_{j} \geq 0, -g_{j}(x,y) \geq 0,$ $j=1,\dots,q$ using $10^{-4}$ as tolerance level, noting that $\mathcal{L}(x,y,u)$ is already included in the systems \eqref{Detailed} and \eqref{Compact}.

\begin{center}
\begin{table}[h]
    \centering
    \begin{tabular}{ccccc}
         \toprule
          & \multicolumn{2}{c}{Experiment I}
          & \multicolumn{2}{c}{Experiment II} \\
          \cline{2-3} \cline{4-5}  \\
          & Detailed & Compact &  Detailed & Compact  \\
   \midrule
   Av. outer iter & 6.9 & 4.7 & 6.12 & 5.56  \\
   Av. time & 0.64 & 0.22 & 0.56 & 0.49 \\
   Av inner iter & 505.23 & 189.27 & 370.9 & 368.4  \\
   Av. accuracy & 0.40 & 0.43 &  0.48 & 0.48 \\
   C-stationarity (\%) & 20 & 43.3 & 8 & 40  \\
   Feasibility (\%) & 43.3 & 83.3 & 18 & 51 \\
   EOC $\leq$ 1  (\%)  & 73.3 & 96.67 & 85 & 86  \\
   EOC  $>$ 1  (\%) & 26.7 & 3.33 & 15 & 14 \\
   \bottomrule
    \end{tabular}
    \caption{Computational result for Scholtes relaxation }
    \label{tab:Scholtes}
\end{table}
\end{center}

{{Based on our numerical experiments, we observed that the compact form of the reformulation of the Scholtes method demonstrates better numerical performance compared to its detailed form. Figure \ref{fig1} illustrates the performance profile of the first experiment, using the number of outer iterations and execution time. Additionally, the bar chart in Figure \ref{fig2} shows the performance of both detailed and compact formulations regarding feasibility checks and C-stationarity checks for each example. Specifically, the compact Scholtes method achieved the least outer iterations numbers in approximately 90\% of the instances while the detailed Scholtes method has the least outer iteration numbers in approximately 10\% instances. Similarly, the compact Scholtes method has the least execution time in approximately 60\% instances while the detailed Scholtes form has the least execution time in approximately 40\% instances. This behavior is corroborated with the average values presented in Table \ref{tab:Scholtes}, where the compact form consistently exhibits the lowest numbers of iterations and execution time.

Furthermore, we can compare the solutions obtained by both detailed and compact Scholtes methods with the solutions reported in \cite{WiesemannTsoukalasKleniatiRustem2013} (which are referred to as the known solution) for each of the problems in the experiments. The optimal function value for the pessimistic problem of mb\_1\_1\_06 as reported in \cite{WiesemannTsoukalasKleniatiRustem2013} is 0 which match the results obtained by both the detailed and compact Scholtes forms in 6 out of 10 instances. Also, the known optimal functional value for both optimistic and pessimistic for problem mb\_1\_1\_10 is 0.1875. The detailed and compact forms of the Scholtes method both obtained this value in 8 out of 10 instances. Additionally, 5 out of 10 instances satisfies the C-stationarity check.  For problem mb\_1\_1\_17, the known optimal objective value reported in \cite{WiesemannTsoukalasKleniatiRustem2013} is -1.7550 for the optimistic case of the problem and -0.2929 for the pessimistic case of the problem. In our experiment, both the detailed and compact Scholtes forms obtained the optimistic solution (-1.7550) in 7 out of 10 instances and the pessimistic value (-0.2929) in 2 out of 10 instances. However, none of the 10 instances satisfied the C-stationarity check in the experiment. Similarly in experiment II, the compact form of the Scholtes method demonstrated better perfomance compared to the detailed form. Specifically, the compact system attained the least number of iterations in approximately 75\% of the test problems, whereas the detailed form managed this in about 60\% of the test problems. Note that this result include the instances where both methods have the same number of iterations. Although the detailed form had a slightly better performance in terms of execution time, with 52\% compared to 48\% for the compact form, the average execution time was smaller for the compact Scholtes form (0.49) than the detailed Scholtes form (0.56). This indicates that the values of the execution time at instances where the detailed form has an advantage are very large compared to the values of the execution time at instances where the compact form has an advantage.

We also assessed whether the solutions computed by the Scholtes methods are the same with the known solutions in \cite{WiesemannTsoukalasKleniatiRustem2013} for the problems. For problem mb\_1\_1\_03, the known optimal solution for both the optimistic and pessimistic cases of the problem is 0.5. This value was obtained by the detailed form in 6 out of 4 instances and by the compact form in 8 out of 10 instances. Additionally, the feasibility condition was met in 4 instances by the detailed form and in all 10 instances by the compact form. However, the C-stationarity condition was not met in any of the 10 instances by either form. More so, neither the detailed nor the compact form was able to compute the optimal function value for the problems mb\_1\_1\_05, mb\_1\_1\_07, mb\_1\_1\_08, mb\_1\_1\_11, and mb\_1\_1\_13 in the 10 instances tested. For the problem mb\_1\_1\_12, the solution for the optimistic problem was obtained in 7 instances by the detailed form and 9 instances out of 10 by the compact form, but the solution for the pessimistic case of the problem was not obtained. On the other hand, both methods performed very well for the problems mb\_1\_1\_09 and mb\_1\_1\_14 obtaining both the optimistic and pessimistic cases of the problems in 8 and 5 instances, respectively. The compact form of the Scholtes method also excelled in satisfying the C-stationarity condition more frequently than the detailed form in each example. Overall, these results underscore the numerical advantages of the compact form of the Scholtes method over the detailed form, particularly in terms of iteration count, execution time, ability to meet feasibility and C-stationarity conditions across a range of test problems. It also support the optimal solution for the test problems as reported in \cite{WiesemannTsoukalasKleniatiRustem2013}.}}

\section{Conclusion and topics for future work}\label{sec:conclusions}
We have considered the KKT reformulation of the pessimistic bilevel optimization problem \eqref{MPCC} and built a theoretical framework to solve it by iteratively computing a sequence of solutions of the relaxation problem (\ref{RMPCCt}) for $t:=t^k$ as $t^k\downarrow 0$. Then, considering the fact that problem \eqref{PBP} is globally/locally equivalent to problem (\ref{MPCC}), under mild assumptions, solving the corresponding class of the pessimistic bilevel program boils down to solving a special class of min-max optimization problems. Solving min-max optimization problems of the type in (\ref{RMPCCt})  is not an easy task, especially as the inner constraints there are parameterized by $(y, u)$ even as we fix $t>0$; a detailed analysis of the complexity of solving such problems is conducted in the recent paper \cite{Daskalakis}. Hence, we also pay attention to the convergence of the Scholtes relaxation in the case where the stationary points of problem \eqref{RMPCCt} are computed and show that the limit of the corresponding sequence is a C-stationary points problem \eqref{MPCC} under suitable assumptions.

Considering techniques from \cite{FischerZemkohoZhou,FliegeTinZemkoho,TinZemkohoLevenberg,ZemkohoZhouTheoretical}, for example, our numerical experiments to compute C-stationary points for \eqref{MPCC} reveal that the compact form of the Scholtes relaxation leads to more practical solutions than the detailed form. It in fact comes out that the compact form of the method takes lesser number of iterations and time of execution to obtain a solution than the detailed form of the Scholtes relaxation algorithm. Moreover, the dimension of the compact form is smaller than that of the detailed form because of the substitution of variables.  In the future, instead of using MATLAB's \texttt{fsolve} function to solve the system of equations, we can explore the use of the semismooth Newton method to solve the system of equations.  Another way to extend our discussion from the present paper is to use the concept of Newton-differentiability  \cite{Harder} instead of the semismooth as the underlying tool of generalized differentiation in the algorithm.

Sufficient conditions ensuring that a point satisfying \eqref{Er1-f1}--\eqref{Er1-f2} and  \eqref{Er1-f3}--\eqref{Er5} could lead to a point that satisfies condition \eqref{Er0} will also be studied in the future.

\section*{Acknowledgments}
The work of the first and third authors is supported by the University of Oran 1 under the PRFU grant with reference 48/S.D.R.F/2023, while the second and fourth authors are partly funded by the EPSRC project with reference EP/V049038/1.
The authors would also like to thank an anonymous referee for carefully reading the paper, and for their thoughtful comments, which  helped in the improvement of the previous version of the paper.

\section*{Data availability statement}
The test problems used for the experiments in this paper can be found in \cite{Mitsos2006,WiesemannTsoukalasKleniatiRustem2013,ZhouZemkohoTin2018}. As for the codes used for the experiments, they are based on MATLAB’s fsolve and can be requested by an email to the authors.

\section*{Conflict of interest statement} The authors have no conflicts of interest to declare.


\section{Appendices}
\subsection{Proof of Proposition \ref{lim}}\label{Proof of Proposition lim}
\noindent  If such a neighborhood does not exist, then we can find a sequence of $(x^{k}, y^{k}, u^{k})_k$, which converges to $(\bar{x}, \bar{y}, \bar{u})$ as $k\rightarrow\infty$  as $k\rightarrow\infty$ with $(y^{k},u^{k})\in \mathcal{D}(x^{k})$ and such that condition $(A_{1}^{m})$ or $(A_{2}^{m})$ does not hold at $(x^{k}, y^{k}, u^{k})$. Assuming that $(A_{1}^{m})$ is not satisfied, there exists a sequence
$(\beta^{k},\gamma^{k})_k$ with $\left\Vert
(\beta^{k},\gamma^{k})\right\Vert =1$ such that we have the inclusion $(\beta^{k},\gamma^{k})\in
\Lambda^{em}(x^{k},y^{k},u^{k},0)$;  i.e.,
\begin{equation}
\left\{
\begin{array}
[c]{l}%
\nabla_{x,y}\mathcal{L}(x^{k},y^{k},u^{k})^{\top}\beta^{k}+\nabla g(x^{k},y^{k})^{\top}\gamma^{k}=0,\medskip \\
\nabla_{y}g_{\nu^{k}}(x^{k},y^{k})\beta^{k}=0, \;\; \gamma_{\eta^{k}}^{k}=0,\medskip \\
\left(\gamma_{i}^{k}>0\wedge\nabla_{y}g_{i}(x^{k},y^{k})\beta^{k}>0\right)\vee\gamma_{i}^{k}\nabla_{y}g_{i}(x^{k},y^{k})\beta^{k}=0\; \mbox{ for all } i\in\theta^{k},
\end{array}
\right.  \label{*}
\end{equation}
where $\theta^{k}:=\theta(x^{k},y^{k},u^{k}),$ $\eta^{k}:=\eta(x^{k}
,y^{k},u^{k})$, and $\nu^{k}:=\nu(x^{k},y^{k},u^{k})$. Clearly,
\begin{equation}\label{te}
\theta^{k}\subset\theta,\text{ }\eta^{k}\subset \theta\cup\eta, \text{ and }\nu^{k}\subset
\theta\cup\nu \;\, \mbox{ for all }\;\, k\;\, \mbox{ sufficiently large.}
\end{equation}
As $i\in\eta^{k}$ for all $i\in\eta$, $i\in \nu^{k}$
for all $i\in \nu$ whenever $k$ is sufficiently large, and by setting
\begin{equation}
\bar{\gamma}_{i}^{k}:=\left\{
\begin{array}
[c]{l}%
\gamma_{i}^{k}\text{\ \ if\ \ \ }i\in\theta\cup \nu,\medskip\\
0\text{ \ \ \ otherwise},
\end{array}
\right.  \label{ni}
\end{equation}
the system of equations (\ref{*}) becomes the following one:
\begin{equation}
\left\{
\begin{array}
[c]{l}
\nabla_{x,y}\mathcal{L}(x^{k},y^{k},u^{k})^{\top} \beta^{k}+\nabla g(x^{k},y^{k})^{\top}\bar{\gamma}^{k}=0,\medskip \\
\nabla_{y}g_{\nu}(x^{k},y^{k})\beta^{k}=0,\text{ \ \ }\bar{\gamma}_{\eta}
^{k}=0,\medskip \\
(\bar{\gamma}_{i}^{k}>0\wedge\nabla_{y}g_{i}(x^{k},y^{k})\beta^{k}
>0)\vee\bar{\gamma}_{i}^{k}(\nabla_{y}g_{i}(x^{k},y^{k})\beta^{k}
)=0\; \mbox{ for all } i\in\theta.
\end{array}
\right.  \label{***}
\end{equation}
Now without loss of generality, we may assume that the sequence $(\beta
^{k},\bar{\gamma}^{k})$ converges to $(\beta,\bar{\gamma})$ as
$k\rightarrow+\infty$ with $\left\Vert (\beta,\bar{\gamma})\right\Vert
=1.$ Consequently, as all functions in (\ref{***}) are continuous, we get
\[
\left\{
\begin{array}
[c]{l}%
\nabla_{x,y}\mathcal{L}(\bar{x},\bar{y},\bar{u})^\top \beta+\nabla
g(\bar{x},\bar{y})^\top\bar{\gamma}=0,\\
\nabla_{y}g_{\nu}(\bar{x},\bar{y})\beta=0,\text{ \ }\bar{\gamma}_{\eta}=0,\\
(\bar{\gamma}_{i}\geq0\wedge\nabla_{y}g_{i}(\bar{x},\bar{y})\beta
\geq0)\vee\bar{\gamma}_{i}(\nabla_{y}g_{i}(\bar{x},\bar{y})\beta
)=0\; \mbox{ for all } i\in\theta.
\end{array}
\right.
\]
Now if for some $i\in\theta,$ $\bar{\gamma}_{i}(\nabla_{y}g_{i}(\bar{x}
,\bar{y})\beta)\neq0$ then $\bar{\gamma}_{i}>0$ and $\nabla_{y}
g_{i}(\bar{x},\bar{y})\beta>0$ and so
\[
(\bar{\gamma}_{i}>0\wedge\nabla_{y}g_{i}(\bar{x},\bar{y})\beta
>0)\vee\bar{\gamma}_{i}(\nabla_{y}g_{i}(\bar{x},\bar{y})\beta
)=0\; \mbox{ for all } i\in\theta.
\]
Hence $(A_{1}^{m})$ is not satisfied at $(\bar{x},\bar{y},\bar{u})$ since
$0\neq(\beta,\bar{\gamma})\in\Lambda^{em}(\bar{x},\bar{y},\bar{u},0).$

Similarly, if $(A_{2}^{m})$ is not satisfied, we can find a nonvanishing
sequence $(\beta^{k},\gamma^{k})_k$ such that
\[
\nabla_{x}\mathcal{L}(x^{k},y^{k},u^{k})^\top\beta^{k}+\nabla_{x}g(x^{k},y^{k})^\top\gamma^{k}\neq0
\]
and $(\beta^{k},\gamma^{k})\in\Lambda_{y}^{em}(x^{k},y^{k},u^{k},0)$ so that,
\begin{equation*}
\left\{
\begin{array}
[c]{l}%
\nabla_{y}\mathcal{L}(x^{k},y^{k},u^{k})^{\top}\beta^{k}+\nabla_{y} g(x^{k},y^{k})^{\top}\gamma^{k}=0,\medskip \\
\nabla_{y}g_{v^{k}}(x^{k},y^{k})\beta^{k}=0, \;\; \gamma_{\eta^{k}}^{k}=0,\medskip \\
\left(\gamma_{i}^{k}>0\wedge\nabla_{y}g_{i}(x^{k},y^{k})\beta^{k}>0\right)\vee\gamma_{i}^{k}\nabla_{y}g_{i}(x^{k},y^{k})\beta^{k}=0\; \mbox{ for all } i\in\theta^{k}.
\end{array}
\right.  \label{**}
\end{equation*}

Taking the sequence $(\bar{\gamma}^{k})_k$ as in (\ref{ni}), we obtain
\begin{equation*}
\left\{
\begin{array}
[c]{l}%
\nabla_{y}\mathcal{L}(x^{k},y^{k},u^{k})^\top\beta^{k}+\nabla_{y}g(x^{k},y^{k})^\top\bar{\gamma}^{k}=0,\medskip\\
\nabla_{y}g_{\nu}(x^{k},y^{k})\beta^{k}=0,\text{ \ \ }\bar{\gamma}_{\eta}^{k}=0,\medskip\\
(\bar{\gamma}_{i}^{k}>0\wedge\nabla_{y}g_{i}(x^{k},y^{k})\beta^{k}>0)\vee\bar{\gamma}_{i}^{k}(\nabla_{y}g_{i}(x^{k},y^{k})\beta^{k}
)=0\; \mbox{ for all } i\in\theta,
\end{array}
\right. \label{**}%
\end{equation*}
and assuming that the sequence $(\beta^{k},\bar{\gamma}^{k})_k$ (up to a
subsequence) converges to the vector $(\beta,\bar{\gamma})$ as $k\rightarrow
+\infty$ with $\left\Vert (\beta,\bar{\gamma})\right\Vert =1,$ we clearly arrive at
\[
\left\{
\begin{array}
[c]{l}%
\nabla_{y}\mathcal{L}(\bar{x},\bar{y},\bar{u})^\top\beta+\nabla_{y}g(\bar{x},\bar{y})^\top\bar{\gamma}=0,\\
\nabla_{y}g_{\nu}(\bar{x},\bar{y})\beta=0,\text{ \ }\bar{\gamma}_{\eta}=0,\\
(\bar{\gamma}_{i}\geq0\wedge\nabla_{y}g_{i}(\bar{x},\bar{y})\beta\geq0)\vee\bar{\gamma}_{i}(\nabla_{y}g_{i}(\bar{x},\bar{y})\beta
)=0\; \mbox{ for all } i\in\theta,
\end{array}
\right.
\]
so that $(\beta,\bar{\gamma})\in\Lambda_{y}^{em}(\bar{x},\bar{y},\bar
{u},0)$; this implies, by $(A_{2}^{m})$, that
\[
\nabla_{x}\mathcal{L}(\bar{x},\bar{y},\bar{u})^{T}\beta+\nabla_{x}
g(\bar{x},\bar{y})^{T}\bar{\gamma}=0.
\]
Hence $(\beta,\bar{\gamma})\in\Lambda^{em}(\bar{x},\bar{y},\bar{u},0),$
leading to $(\beta,\bar{\gamma})=0$, by $(A_{1}^{m})$ at $(\bar
{x},\bar{y},\bar{u})$.  
\hfill \qed

\subsection{Proof of Theorem \ref{TheoremFinal}}\label{Proof of Theorem TheoremFinal}
\noindent Since $\bar{x}$ is upper-level regular, by continuity of $G$, the upper-level regularity condition \eqref{UMFCQ} is satisfied at all points $x\in X$, which are sufficiently close to $\bar{x}$. Now from Proposition \ref{lim}, there is a neighborhood $U\times V$ of
$(\bar{x},(\bar{y},\bar{u}))$ such that $(A_{1}^{m})$ and $(A_{2}^{m})$ hold
at any $(x,(y,u)) \in U\times V$. Consider now $t$ close
to $0$ and $(x,(y,u))$ sufficiently close to $(\bar{x},(\bar{y},\bar{u})),$
namely$,$ $x\in U\cap X$ and $(y,u)\in V\cap\mathcal{D}^{t}(x).$ Let
$(\beta,\gamma,\delta,\mu)\in\mathbb{R}^{m+3q}$ such that
\[
\left\{
\begin{array}
[c]{l}%
\nabla_{y}\mathcal{L}^{T}(x,y,u)\beta+\sum\limits_{i=1}^{q}(\gamma_{i}%
-\delta_{i}u_{i})\nabla_{y}g_{i}(x,y)=0,\\
\nabla_{y}g_{i}(x,y)\beta-\mu_{i}-\delta_{i}g_{i}(x,y)=0, \;\, i=1,...,q,\\[1ex]
\gamma_{i}\geq0,\text{ }\delta_{i}\geq0,\text{ }\mu_{i}\geq0, \;\, i=1,...,q,\\[1ex]
\mu_{i}u_{i}=0,\text{ \ }\gamma_{i}g_{i}(x,y)=0,\text{ \ }\delta_{i}(u_{i}g_{i}(x,y)+t)=0, \;\, i=1,...,q.
\end{array}
\right.
\]
Hence, it follows that we have
\[
\nabla_{y}g_{i}(x,y)\beta=0 \, \mbox{ for }\, i\in \nu(x,y,u),\, \gamma_{i}=0 \, \mbox{ for }\, i\in \eta(x,y,u), \, \mbox{ and }\, \nabla_{y}g_{i}(x,y)\beta=\mu_{i}\, \mbox{ for }\, i\in \theta(x,y,u).
\]
Now setting
\[
\tilde{\gamma}_{i}:=\left\{
\begin{array}{lll}
\gamma_{i}       & \;\;\mbox{ if } \;\;& i\in\mathrm{supp}\gamma,\\
-\delta_{i}u_{i} & \;\;\mbox{ if } \;\; & i\in\mathrm{supp}\delta, \\
  0               & \;\; \mbox{ if } \;\;& \mbox{otherwise},
\end{array}
\right.
\]
leads to the conditions
\[
\nabla_{y}\mathcal{L}^{T}(x,y,u)\beta+\sum\limits_{i=1}^{q}\tilde{\gamma
}i\nabla_{y}g_{i}(x,y)=0 \;\; \mbox{ and }\;\;
\left\{\begin{array}{lll}
  \nabla_{y}g_{i}(x,y)\beta=0 & \mbox{ for } & i\in \nu(x,y,u),\\
 \tilde{\gamma}_{i}=0 & \mbox{ for } & i\in \eta(x,y,u),\\
  \tilde{\gamma}_{i}\nabla_{y}g_{i}(x,y)\beta=\tilde{\gamma}_{i}\mu_{i} & \mbox{ for } & i\in \theta(x,y,u),
\end{array}\right.
\]
and as $i\notin\mathrm{supp}\delta$ whenever $i\in\theta(x,y,u)$ so that for all $i\in\theta(x,y,u)$,
\[
0\leq \tilde{\gamma}_{i}\nabla_{y}g_{i}(x,y)\beta
\;\;=\;\;\left\{\begin{array}{lll}
          \gamma_{i}\mu_{i} & \mbox{ if } & i\in\mathrm{supp}\gamma \cap \mathrm{supp}\mu ,\\
          0                 &             & \mbox{otherwise}.
        \end{array}\right.
\]
Thus, for any $i\in\theta(x,y,u)$, we have
\[
(\tilde{\gamma}_{i}=\gamma_{i}>0\wedge\nabla_{y}g_{i}(x,y)\beta=\mu_{i}%
>0)\vee\tilde{\gamma}_{i}\nabla_{y}g_{i}(x,y)\beta=0
\]
so that $(\beta,\tilde{\gamma})\in\Lambda_{y}^{em}(x,y,u,0)$. Applying
$(A_{2}^{m})$, we get
$
\nabla_{x}\mathcal{L}(x,y,u)^{T}\beta+\nabla_{x}g(x,y)^{T}\tilde{\gamma}=0
$
and thus, we have  $(\beta,\tilde{\gamma})\in\Lambda^{em}(x,y,u,0)$, which implies, by
$(A_{1}^{m})$, that $\beta=0$ and $\tilde{\gamma}=0.$ \hfill \qed

\end{document}